\documentclass{amsart}

\usepackage{comment}
\usepackage{hyperref}
\def\NZQ{\mathbb}               
\def\NN{{\NZQ N}}

\def\RR{{\NZQ R}}

\def\F2{{\NZQ F}_2}

\def\mm{{\frk m}}

\def\opn#1#2{\def#1{\operatorname{#2}}} 
\opn\chara{char} \opn\length{\ell} \opn\pd{pd} \opn\rk{rk}
\opn\projdim{proj\,dim} \opn\injdim{inj\,dim} \opn\rank{rank}
\opn\depth{depth} \opn\codepth{codepth} \opn\grade{grade}
\opn\height{height} \opn\embdim{emb\,dim} \opn\codim{codim}

\opn\Tr{Tr} \opn\bigrank{big\,rank}
\opn\superheight{superheight}\opn\lcm{lcm}
\opn\trdeg{tr\,deg}%
\opn\reg{reg} \opn\lreg{lreg} \opn\skel{skel}
\opn\Gr{Gr}
\opn\ann{ann}
\opn\sign{sign}
\opn\del{del}
\opn\lex{lex}

\opn\div{div} \opn\Div{Div} \opn\cl{cl} \opn\Cl{Cl}
\opn\Spec{Spec} \opn\Supp{Supp} \opn\supp{supp} \opn\Sing{Sing}
\opn\Ass{Ass}\opn\fdepth{fdepth}
\opn\Ann{Ann} \opn\Rad{Rad} \opn\Soc{Soc}
\opn\Sym{Sym} \opn\Ker{Ker} \opn\Coker{Coker} \opn\Im{Im}
\opn\Hom{Hom} \opn\Tor{Tor} \opn\Ext{Ext} \opn\End{End}
\opn\Aut{Aut} \opn\id{id} \opn\ini{in} \opn\tr{tr}

\opn\nat{nat}\opn\it{it}
\opn\pff{proof}
\opn\Pf{proof} \opn\GL{GL} \opn\SL{SL} \opn\mod{mod} \opn\ord{ord}
\opn\aff{aff} \opn\con{conv} \opn\relint{relint} \opn\st{st}
\opn\lk{lk} \opn\cn{cn} \opn\core{core} \opn\vol{vol}
\opn\link{link} \opn\star{star} \opn\skel{skel} \opn\indeg{indeg}
\opn\Ass{Ass} \opn\Min{Min} \opn\sdepth{sdepth} \opn\depth{depth}
\opn\gr{gr}

\def\pot#1#2{#1[\kern-0.28ex[#2]\kern-0.28ex]}

\opn\dirlim{\underrightarrow{\lim}}
\opn\inivlim{\underleftarrow{\lim}}
\let\iso=\cong

\let\to=\rightarrow

\def\Implies{\ifmmode\Longrightarrow \else
     \unskip${}\Longrightarrow{}$\ignorespaces\fi}
\def\implies{\ifmmode\Rightarrow \else
     \unskip${}\Rightarrow{}$\ignorespaces\fi}
\def\iff{\ifmmode\Longleftrightarrow \else
     \unskip${}\Longleftrightarrow{}$\ignorespaces\fi}

\let\:=\colon
\theoremstyle{plain}
\newtheorem{Theorem}{Theorem}[section]
 \newtheorem{Lemma}[Theorem]{Lemma}
 \newtheorem{Corollary}[Theorem]{Corollary}
 \newtheorem{Proposition}[Theorem]{Proposition}
 
 \newtheorem{Conjecture}[Theorem]{Conjecture}
 
 \theoremstyle{definition}
 \newtheorem{Definition}[Theorem]{Definition}
 \newtheorem{Remark}[Theorem]{Remark}
 
 \newtheorem{Example}[Theorem]{Example}

\let\epsilon\varepsilon
\let\kappa=\varkappa
\opn\dis{dis}
\def\pnt{{\raise0.5mm\hbox{\large\bf.}}}

\opn\Lex{Lex}

\newcommand{\rad}{1.5 pt}
\usepackage{subcaption,tikz}
\usetikzlibrary{patterns}

\usepackage{float}
\usepackage{tikz-cd}
\usepackage{mathtools}
\usepackage{amsfonts}
\usepackage{amssymb}
\newcommand{\PP}{\mathcal{P}}

\renewcommand{\H}{\mathrm{H}}
\newcommand{\Pc}{\mathcal{P}}
\newcommand{\HS}{\mathrm{HS}}
\renewcommand{\RR}{\mathcal{R}}
\renewcommand{\SS}{\mathcal{S}}

\newcommand{\MM}{\mathcal{M}}

\renewcommand{\mm}{\mathfrak{m}}

\title{Hilbert series of Parallelogram Polyominoes}
\author[Ayesha Asloob Qureshi]{Ayesha Asloob Qureshi}

\address[Ayesha Asloob Qureshi]{Sabanc\i \; University, Faculty of Engineering and Natural Sciences, Orta Mahalle, Tuzla 34956, Istanbul, Turkey}\email{aqureshi@sabanciuniv.edu}

\author[Giancarlo Rinaldo]{Giancarlo Rinaldo}

\address[Giancarlo Rinaldo]{Department of Mathematics, Informatics, Physics and Earth science,\\
University of Messina\\
Viale F. Stagno d’Alcontres, 31\\
98166 Messina, Italy}
\email{giancarlo.rinaldo@unime.it}

\author[Francesco Romeo]{Francesco Romeo}

\address[Francesco Romeo]{Department of Mathematics\\
University of Trento\\
via Sommarive, 14\\
38123 Povo (Trento), Italy}
\email{francesco.romeo-3@unitn.it}

\thanks{This work was supported by The Scientific and Technological Research Council of Turkey - TUBITAK (Grant No: 118F169).}
\dedicatory{Dedicated to J\"urgen Herzog on the occasion of his 80th birthday}

\begin{document}
\maketitle

\begin{abstract}
We present a conjecture about the reduced Hilbert series of the coordinate ring of a simple polyomino in terms of particular arrangements of non-attacking rooks that can be placed on the polyomino. By using a computational approach, we prove that the above conjecture holds for all simple polyominoes up to rank $11$. In addition, we prove that the conjecture holds true for the class of parallelogram polyominoes, by looking at those as simple planar distributive lattices. Finally, we give a combinatorial interpretation of the Gorensteinnes of parallelogram polyominoes. 

\end{abstract}

\section{Introduction}
Polyominoes are two dimensional objects obtained by joining  squares of the same size edge to edge. They are originally rooted in recreational mathematics and combinatorics and have been widely discussed in connection with tiling and enumeration problem, see \cite{Go}. Certain special classes of polyominoes are also related to algebraic languages. In \cite{DV}, the relation of polyominoes with Dyck words and Motzkin words is beautifully elaborated. 

In 2012, the first author introduced a  quadratic binomial ideal associated to polyominoes, the \emph{polyomino ideal} $I_\PP$, see \cite{Qu}. Let $K$ be a field and $X$ be a $m \times n$ matrix of indeterminates. The polyomino ideals are generated by quite general sets of 2-minors of $X$, and they include some other well known classes of binomial ideals such as the ideals generated by 2-minors of ladders and the join-meet ideals of distributive lattices. The motivation of building such a connection between polyominoes and binomial ideal is to facilitate the study of the binomial ideals and to translate their algebraic properties in combinatorial aspects of polyominoes. Let $I_{\Pc}$ be the polyomino ideal associated to a polyomino $\Pc$ and $K[\Pc]$ be the associated coordinate ring $K[\Pc]$ (see \cite{Qu}). An open and challenging task  is to characterize all polyominoes (in terms of their shape) whose associated coordinate ring is a domain. In \cite{MRR}, authors proved that if $K[\Pc]$ is a domain then $\Pc$ must avoid the ``zig-zag walks". The converse of this statement is an open question. It is known from \cite{HM,QSS} that the associated coordinate rings of simple polyominoes are domains. Roughly speaking, simple polyominoes are polyominoes without holes. 

The contents of this paper are arranged as follows: in Section~\ref{sec:pre}, we recall basic notion and definitions related to polyominoes and distributive lattice. In Subsection~\ref{subsec:polyo-lattice}, the relationship between polyominoes and distributive lattices is explored. In particular, in Proposition~\ref{prop:parallel=simplelattice}, we prove that the parallelogram polyominoes are simple planar distributive lattices. The parallelogram polyominoes are widely discussed in combinatorics (see \cite{DV,DNPR,ABG} for problems related to enumeration and the computation of  the total area),  and they are defined by two lattice paths that use north and east unit steps, and intersect only at their origin and extremity. Let $L$ be a distributive lattice. Given any pair of incomparable elements $\alpha, \beta \in L$,  let $f_{\alpha\beta}$ be the binomial $z_{\alpha \vee \beta}z_{\alpha \wedge \beta}-z_\alpha z_\beta$ in the polynomial ring $K[z_\alpha : \alpha \in L]$. The ideal generated by all such binomials $f_{\alpha\beta}$ is called the \emph{join-meet ideal} of $L$, and is denoted by $I_L$. Moreover,  $K[L]=K[z_\alpha : \alpha \in L]/I_L$ is known as the \emph{Hibi ring} of $L$. The Hibi rings are well understood and possess several nice algebraic and homological properties. It is known from \cite{H} that Hibi rings are Cohen-Macaulay domains and the generators of $I_L$ form the reduced Gr\"obner basis with respect to reverse lexicographical order. Moreover, the Hilbert series of Hibi rings is described in \cite{BGS}. The Hibi ring of simple planar distributive lattice coincide with its coordinate ring as a polyomino. In particular, the join-meet ideal and polyomino ideal of a simple planar distributive lattice is the same, see Remark~\ref{rem:sameideal}.  This  identification allows us to use the existing knowledge on Hibi rings arising from simple planar distributive lattices and translate it in terms of their structure as coordinate rings of parallelogram polyominoes. Moreover, from \cite[Theorem 3.1]{EHQR} it is known that the polyomino ideals of $L$-convex polyominoes can be interpreted as polyomino ideals of  certain Ferrer diagrams. The Ferrer diagrams are a special subclass of parallelogram polyominoes. Therefore, the results provided in subsequent sections hold for $L$-convex polyominoes, which in general, do not have a structure of a simple planar distributive lattice.

In Section~\ref{sec:rookcomp}, we study the Hilbert series of parallelogram polyominoes. A polyomino can be viewed as a pruned chessboard. With this point of view, in combinatorics the rook polynomial of polyominoes is a well studied topic, for example see  \cite[Chapter 7]{Ri}. A rook polynomial $\sum_{i=1}^n r_ix^i$ is a polynomial whose coefficients $r_i$ represents the number of distinct ways of arranging $i$ rooks on squares of $\Pc$ in non-attacking positions.  In \cite{EHQR} and \cite{RR}, the authors linked the Castelnuovo-Mumford regularity of $K[\Pc]$ to the maximum number of non-attacking rooks that can be placed on the polyomino, for the classes of $L$-convex polyominoes and simple thin polyominoes. For the latter class, in \cite{RR}, the authors proved that the polynomial $h(t)$ of the reduced Hilbert series $h(t)/(1-t)^d$ of $K[\Pc]$ coincides with the \emph{rook polynomial} of $\Pc$. This result in \cite{RR} motivated us to study the relation between the Hilbert series and the rook polynomial for simple non-thin polyominoes.  Recently, another paper in this direction has been written by Kummini and Veer \cite{KV}. In Section~\ref{sec:rookcomp}, we introduce an equivalence relation on the rook complex of a simple polyomino $\Pc$. We conjecture that the number of equivalence classes of $k$ non-attacking rooks arrangements coincides with the coefficient $h_k$ of the polynomial $h(t)$ in the reduced Hilbert series. We prove that Conjecture \ref{conj:rtilde} holds true for the class of \emph{parallelogram polyominoes}. Moreover, by using a computational approach, we prove that Conjecture \ref{conj:rtilde} holds true for any simple polyomino having at most 11 cells. In \cite{QRR}, we provide an implementation in \texttt{Macaulay2} \cite{M2} and Java for such computations.

The Gorensteinness of the coordinate ring of some classes of polyominoes has been studied by many authors, for example see \cite{Qu}, \cite{EHQR} and \cite{RR}.
Even though the Gorenstein ladder determinantal rings and the Gorenstein Hibi rings are completely characterized, in Section~\ref{sec:gor} we give a combinatorial characterization of Gorenstein parallelogram polyominoes that is analogous to the characterizations given in \cite{EHQR} and \cite{RR} for $L$-convex and simple thin polyominoes, respectively. Such characterization involves the intersections of the maximal rectangles of the parallelogram polyominoes. It is well-known that every parallelogram polyomino can be uniquely represented as a Motzkin path. In Corollary \ref{cor:Motz}, we give a characterization of the Motzkin paths which represent Gorenstein parallelogram polyominoes.

\section{Distributive Lattices and Polyominoes}\label{sec:pre}

\subsection{Polyominoes and polyomino ideals}
In this subsection, we recall general definitions and notation on polyominoes. 

Let $a = (i, j), b = (k, \ell) \in \NN^2$, with $i	\leq k$ and $j\leq\ell$. The set $[a, b]=\{(r,s) \in \NN^2 : i\leq r \leq k \text{ and } j \leq s \leq \ell\}$ is called an \textit{interval} of $\NN^2$. Moreover, if $i<k$ and $j < \ell$, then $[a,b]$ is called a \textit{proper interval}, and the elements $a,b,c,d$ are called corners of $[a,b]$, where $c=(i,\ell)$ and $d=(k,j)$. In particular, $a,b$ are the \textit{diagonal corners} and $c,d$ are the \textit{anti-diagonal corners} of $[a,b]$. The corner $a$ (resp. $c$) is also called the left lower (resp. upper) corner of $[a,b]$, and $d$ (resp. $b$) is the right lower (resp. upper) corner of $[a,b]$. A proper interval of the form $C = [a, a + (1, 1)]$ is called a \textit{cell}. The corners of $C$ are called the vertices of $C$. The set of vertices of $C$ is denoted by $V(C)$. The edge set of $C$, denoted by $E(C)$, is 
\[
 \{\{a,a+(1,0)\}, \{a,a+(0,1)\},\{a+(1,0),a+(1,1)\},\{a+(0,1),a+(1,1)\} \}.
\]
We denote by $\ell (C)$, the left lower corner of a cell $C$.  

Let $\PP$ be a finite collection of cells of $\NN^2$, and let $C$ and $D$ be two cells of $\PP$. Then $C$ and $D$ are said to be \textit{connected}, if there is a sequence of cells $C = C_1,\ldots, C_m = D$ of $\PP$ such that $C_i\cap C_{i+1}$ is an edge of $C_i$
for $i = 1,\ldots, m - 1$. In addition, if $C_i \neq C_j$ for all $i \neq j$, then $C_1,\dots, C_m$ is called a \textit{path} (connecting $C$ and $D$). A collection of cells $\PP$ is called a \textit{polyomino} if any two cells of $\PP$ are connected. We denote by $V(\PP)=\cup _{C\in \PP} V(C)$ the vertex set of $\PP$ and by $E(\PP)=\cup _{C\in \PP} E(C)$ the edge set of $\PP$. In particular, a polyomino could be also seen as a connected bipartite graph. Note that, if $a,b\in V(\PP)$, then $a$ and $b$ are connected in $V(\PP)$ by a path of edges. More precisely,  one can find  a sequence of vertices $a=a_1,\ldots, a_{n}=b$ such that $\{a_i, a_{i+1}\} \in E(\PP)$, for all $i=1, \ldots, n-1$.  The number of cells of $\PP$ is called the \textit{rank} of $\PP$, and we denote it by $\rk \PP$. We also define the \emph{lower left corner} of $\PP$ as $\ell (\PP)= \min \{ \ell (C) : C \in \PP \}$. Each proper interval  $[(i,j),(k,l)] $ in $\NN^2$ can be identified as a polyomino and it is referred to as {\em rectangular} polyomino, or simply as rectangle. If $s=k-i$ and $t=l-j$ we say that the rectangle has {\em size} $s \times t$. In particular, given a rectangle of $\PP$ we call \emph{diagonal cells} the cells $A,B$ such that $\ell(A)=(i,j)$ and $\ell(B)=(k-1,l-1)$ and \emph{antidiagonal cells} the cells $C,D$ such that $\ell(C)=(i,l-1)$ and $\ell(D)=(k-1,j)$.

A polyomino $\PP$ is called a \emph{subpolyomino} of $\PP'$, if all cells of $\PP$ are contained in $\PP'$. Given a polyomino $\PP$, the smallest rectangle (with respect to its size) containing $\PP$ as a subpolyomino, is called the \emph{bounding box} of $\PP$.

A proper interval $[a,b]$ is called an \emph{inner interval} of $\PP$ if all cells of $[a,b]$ belong to $\PP$.
We say that a polyomino $\PP$ is \textit{simple} if for any two cells $C$ and $D$ of $\NN^2$ not belonging to $\PP$, there exists a path $C=C_{1},\dots,C_{m}=D$ such that $C_i \notin \PP$ for any $i=1,\dots,m$. Roughly speaking, a polyomino without a ``hole'' is called a simple polyomino. 
An interval $[a,b]$ with $a = (i,j)$ and $b = (k, \ell)$ is called a \emph{horizontal edge interval} of $\PP$ if $j =\ell$ and the sets $\{(r, j), (r+1, j)\}$ for
$r = i, \dots, k-1$ are edges of cells of $\PP$. If a horizontal edge interval of $\PP$ is not strictly contained in any other horizontal edge interval of $\PP$, then we call it \emph{maximal horizontal edge interval}. Similarly, one defines vertical edge intervals and maximal vertical edge intervals of $\PP$. 

A polyomino $\PP$ is called \emph{row convex} if for any two of its cells with lower left corners $a=(i,j)$ and $b=(k,j)$, with $k>i$, all cells with lower left corners $(l,j)$ with $i\leq l \leq k$ are cells of $\PP$. Similarly,  $\PP$ is called \emph{column convex} if for any two of its cells with lower left corners $a=(i,j)$ and $b=(i,k)$, with $k>j$, all cells with lower left corners $(i,l)$ with $j\leq l \leq k$ are cells of $\PP$. If a polyomino $\Pc$ is simultaneously row and column convex then $\Pc$ is called \emph{convex}. Let $\mathcal{C}:C_1, C_2, \ldots, C_m$ be a path of cells and $(i_k, j_k)$ be the lower left corner of $C_k$ for $1 \leq k \leq m$. Then $\mathcal{C}$ has a change of direction at $C_k$ for some $2 \leq k \leq m-1$ if $i_{k-1}\neq i_{k+1}$ and $j_{k-1} \neq j_{k+1}$.
A convex polyomino $\Pc$ is called $k$-convex if any two cells in $\Pc$ can be connected by a path of cells in $\Pc$ with at most $k$ change of directions. The $1$-convex polyominoes are referred to as $L$-convex polyominoes and 2-convex polyominoes are referred to as $Z$-convex polyominoes in literature. Figure \ref{fig:NL2con} shows three examples of polyominoes that are non-simple, $L$-convex and $2$-convex respectively.

\begin{figure}[H]
    \centering
    \begin{subfigure}{0.30 \textwidth}
     \resizebox{0.6\textwidth}{!}{
  \begin{tikzpicture}

\draw[thick] (0,1) --  (0,3);
\draw[thick] (1,1) --  (1,5);
\draw[thick] (2,1) --  (2,5);
\draw[thick] (3,0) --  (3,5);
\draw[thick] (4,0) --  (4,2);

\draw[thick] (0,1) --  (4,1);
\draw[thick] (0,2) --  (4,2);
\draw[thick] (0,3) --  (3,3);
\draw[thick] (1,4) --  (3,4);
\draw[thick] (1,5) --  (3,5);
\draw[thick] (3,0) --  (4,0);

\fill[fill=gray, fill opacity=0.2] (0,1) -- (0,3)-- (1,3) -- (1,1);
\fill[fill=gray, fill opacity=0.2] (2,1) -- (2,5)-- (3,5) -- (3,1);
\fill[fill=gray, fill opacity=0.2] (1,3) -- (1,5)-- (2,5) -- (2,3);
\fill[fill=gray, fill opacity=0.2] (3,0) -- (3,2)-- (4,2) -- (4,0);
\fill[fill=gray, fill opacity=0.2] (1,1) -- (1,2)-- (2,2) -- (2,1);

   \end{tikzpicture}}
    \end{subfigure}\hfill%
    \begin{subfigure}{0.30 \textwidth}
    \resizebox{0.8\textwidth}{!}{
  \begin{tikzpicture}
\draw[thick] (-1,1) --  (-1,2);
\draw[thick] (0,1) --  (0,3);
\draw[thick] (1,0) --  (1,5);
\draw[thick] (2,0) --  (2,5);
\draw[thick] (3,0) --  (3,4);
\draw[thick] (4,1) --  (4,2);

\draw[thick] (-1,1) --  (4,1);
\draw[thick] (-1,2) --  (4,2);
\draw[thick] (0,3) --  (3,3);
\draw[thick] (1,4) --  (3,4);
\draw[thick] (1,5) --  (2,5);
\draw[thick] (1,0) --  (3,0);

\fill[fill=gray, fill opacity=0.2] (1,2) -- (2,2)-- (2,3) -- (1,3);
\fill[fill=gray, fill opacity=0.2] (0,1) -- (0,3)-- (1,3) -- (1,1);
\fill[fill=gray, fill opacity=0.2] (2,0) -- (2,4)-- (3,4) -- (3,0);
\fill[fill=gray, fill opacity=0.2] (1,3) -- (1,5)-- (2,5) -- (2,3);
\fill[fill=gray, fill opacity=0.2] (3,1) -- (3,2)-- (4,2) -- (4,1);
\fill[fill=gray, fill opacity=0.2] (1,1) -- (1,2)-- (2,2) -- (2,1);
\fill[fill=gray, fill opacity=0.2] (1,0) -- (2,0)-- (2,1) -- (1,1);
\fill[fill=gray, fill opacity=0.2] (-1,1) -- (0,1)-- (0,2) -- (-1,2);
\end{tikzpicture}}
    \end{subfigure}\hfill%
    \begin{subfigure}{0.30 \textwidth}
    \resizebox{0.8\textwidth}{!}{
       \begin{tikzpicture}
\draw[thick] (-1,1) --  (-1,2);
\draw[thick] (0,1) --  (0,3);
\draw[thick] (1,0) --  (1,5);
\draw[thick] (2,0) --  (2,5);
\draw[thick] (3,0) --  (3,4);
\draw[thick] (4,3) --  (4,4);

\draw[thick] (-1,1) --  (3,1);
\draw[thick] (-1,2) --  (3,2);
\draw[thick] (0,3) --  (4,3);
\draw[thick] (1,4) --  (4,4);
\draw[thick] (1,5) --  (2,5);
\draw[thick] (1,0) --  (3,0);

\fill[fill=gray, fill opacity=0.2] (1,2) -- (2,2)-- (2,3) -- (1,3);
\fill[fill=gray, fill opacity=0.2] (0,1) -- (0,3)-- (1,3) -- (1,1);
\fill[fill=gray, fill opacity=0.2] (2,0) -- (2,4)-- (3,4) -- (3,0);
\fill[fill=gray, fill opacity=0.2] (1,3) -- (1,5)-- (2,5) -- (2,3);
\fill[fill=gray, fill opacity=0.2] (3,3) -- (3,4)-- (4,4) -- (4,3);
\fill[fill=gray, fill opacity=0.2] (1,1) -- (1,2)-- (2,2) -- (2,1);
\fill[fill=gray, fill opacity=0.2] (1,0) -- (2,0)-- (2,1) -- (1,1);
\fill[fill=gray, fill opacity=0.2] (-1,1) -- (0,1)-- (0,2) -- (-1,2);

\end{tikzpicture}}
    \end{subfigure}
    \caption{From left to right: a non-simple polyomino, an $L$-convex polyomino and a $2$-convex polyomino}\label{fig:NL2con}
\end{figure}
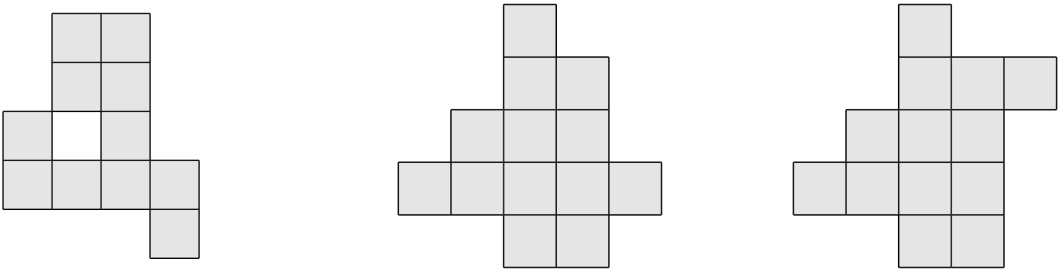

Let $\PP$ be a polyomino and define the polynomial ring $R = K[x_v \ | \ v \in V(\PP)]$ over a field $K$. The binomial $x_a x_b - x_c x_d\in R$ is called an \textit{inner 2-minor} of $\PP$ if $[a,b]$ is an inner interval of $\PP$, where $c,d$ are the anti-diagonal corners of $[a,b]$. The ideal $I_\PP\subset R$ generated by all of the inner $2$-minors of $\PP$ is called the \textit{polyomino ideal} of $\PP$. The quotient ring $K [\PP] = R/I_\PP$ is called the coordinate ring of $\Pc$. It is known from \cite[Theorem 2.1]{HM} and \cite[Corollary 2.3]{QSS} that if $\PP$ is a simple polyomino then $K[\PP]$ is a normal Cohen-Macaulay domain. Combining this with \cite[Corollary 3.3]{HQS}, one obtains the following
\begin{Lemma}\label{lem:dim}
Let $\PP$ be a simple polyomino. Then $K[\PP]$ is a Koszul, normal Cohen-Macaulay domain of Krull dimension $|V(\PP)|-\rk \PP$.
\end{Lemma}

\subsection{Distributive lattices and join-meet ideals}
Let $P$ be a poset with partial order relation $<$. A chain of $P$ is a totally ordered subset of $P$. The length of a chain $\mathfrak{c}$, denoted by $\mathrm{length}(\mathfrak{c})$,  is $|\mathfrak{c}|-1$. Given $a \in P$, the rank of $a$ in $P$, denoted by $\rank(a)$, is the  supremum of length of chains in $P$ that descends from $a$. The {\em rank} of $P$, denoted by $\rank (P)$, is the supremum of length of chains of $P$. An  order ideal of $P$ is a subset $I$ of P with the following property:  if $a \in I$ then $b \in I$ for all $b \in P$ with $b <a$. Two element $a,b \in P$ are called incomparable if $ a \not<b$ and $b \not<a$.
 
 Let $L$ be a distributive lattice with unique minimal element $\min(L)$ and unique maximal element $\max(L)$. An element $a \in L$ is called {\em join-irreducible} if $a \neq \min(L)$ and $a \neq b \vee c$ for any $b,c \in L\setminus \{a\}$. Let $P$ be the set consisting of all join-irreducible elements of $L$. Then $P$ is a poset with partial order inherited from $L$.  Let $I(P)$ be the set consisting of all order ideals of $P$, ordered by inclusion. In particular, $\emptyset, P \in I(P)$ and $I(P)$ is a distributive lattice with $\min(I(P)) = \emptyset$ and $\max(I(P))=P$. It is known by Birkhoff's fundamental theorem  of finite distributive lattices \cite[Chapter 9]{B} that $L \iso I(P)$. Moreover, $\rank(L)=|P|$. We refer to \cite[Chapter 1]{B} for the basic definitions and notation in lattice theory.
 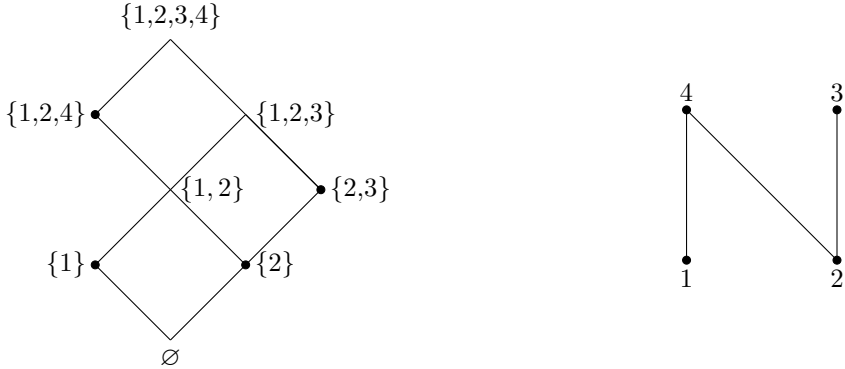
\begin{figure}[H]
 \centering
 \begin{subfigure}{0.5\textwidth}
  \centering
\begin{tikzpicture}
\draw (1,0)--(0,1);
\draw (1,0)--(2,1);
\draw (2,1)--(1,2);
\draw (1,2)--(0,1);
\draw (2,1)--(3,2);
\draw (2,3)--(3,2);
\draw (1,2)--(2,3);
\draw (2,3)--(3,2);
\draw (1,2)--(0,3);
\draw (0,3)--(1,4);
\draw (2,3)--(1,4);
\node at (1,0)[anchor=north]{$\varnothing$};
\filldraw (0,1) circle (\rad) node [anchor=east]{\{1\}};
\filldraw (2,1) circle (\rad) node [anchor=west]{\{2\}};
\filldraw (3,2) circle (\rad) node [anchor=west]{\{2,3\}};
\node at (1,2)[anchor=west]{$\{1,2\}$};
\node at (2,3)[anchor=west]{\{1,2,3\}};
\filldraw (0,3) circle (\rad) node [anchor=east]{\{1,2,4\}};
\node at (1,4)[anchor=south]{\{1,2,3,4\}};
\end{tikzpicture}
\end{subfigure}%
\begin{subfigure}{0.5\textwidth}
 \centering
\begin{tikzpicture}
\draw (0,0)--(0,2);
\draw (2,0)--(2,2);
\draw (2,0)--(0,2);

\filldraw (0,0) circle (\rad) node [anchor=north]{1};
\filldraw (0,2) circle (\rad) node [anchor=south]{4};
\filldraw (2,0) circle (\rad) node [anchor=north]{2};
\filldraw (2,2) circle (\rad) node [anchor=south]{3};

\end{tikzpicture}
\end{subfigure}
\caption{A poset and its ideals lattice $I(P)$}. 
\end{figure}
Let $L$ be finite distributive lattice and $S=K[x_a : a \in L]$. The join-meet ideal $I_L\subset S$ of $L$ is the ideal generated by binomials $x_ax_b-x_{a\vee b} x_{a\wedge b}$, where $a$ and $b$ are incomparable elements in $L$.  It is known from \cite{H} that $K[L]=S/I_L$ is a normal, Cohen-Macaulay domain.
 
 Now we recall some basic properties of planar distributive lattice. Consider the natural partial order on $\NN^2$ defined as follows: for any $(i,j), (k,l) \in \NN^2$, we have $(i,j) \leq (k,l)$ if and only if $i\leq k$ and $j \leq l$. With this natural partial order, $\NN^2$ is an infinite distributive lattice. Let $L$ be a finite sublattice of $\NN^2$. Then $L$ is called a {\em planar} distributive lattice if $(0,0) \in L$ and for any $(i,j), (k,l) \in L$ with$(i,j)< (k,l)$ , there exists a chain in $L$ of the form $(i,j)=(i_0,j_0)< (i_1, j_1)< \ldots < (i_s,j_s)=(k,l)$ such that $i_{k+1}+j_{k+1}= i_k+j_k+1$ for all $k$. The condition $i_{k+1}+j_{k+1}= i_k+j_k+1$ yields that either $(i_{k+1},j_{k+1})=(i_k,j_k)+(0,1)$ or  $(i_{k+1},j_{k+1})=(i_k,j_k)+(1,0)$. A planar distributive lattice $L$ is called {\em simple} if, for all $0 <r<\rank(L)$, there exist at least two elements in $L$ with rank $r$. Equivalently,  $L$ is simple if there is no $a \in L$ with the property that for every $b\in L$ either $a \leq b$ or $b \leq a$.

\subsection{The relationship between polyominoes and distributive lattices}\label{subsec:polyo-lattice}

In this section, we talk about the polyominoes arising from simple planar distributive lattices. Note that any simple planar distributive lattice $L$ can be identified as a convex polyomino.
 
 \begin{Proposition}\label{prop:convex=lattice}
 Let $\Pc$ be a convex polyomino with bounding box $[(0,0), (m,n)]$. If $(0,0), (m,n) \in V(\Pc)$, then $V(\Pc)$ determines a simple planar distributive lattice.
 \end{Proposition}
\begin{proof}
First we show that $V(\Pc)$ is a sublattice of $\NN^2$. Let $a,b \in V(\Pc)$ be two incomparable elements. We need to show that $a \vee b$ and $a\wedge b$ belong to $V(\Pc)$. Let $a=(i,j)$ and $b=(k,l)$. Since $a$ and $b$ are incomparable, we may assume that $i<k$ and $j>l$. Then $a \vee b=(k,j)$ and $a\wedge b=(i,l)$.  First we claim that $a\wedge b=(i,l) \in V(\Pc)$. On the contrary, suppose that $a\wedge b=(i,l) \notin V(\Pc)$. By using the convexity of $\Pc$ and applying \cite[Lemma 1.1]{Qu}, it follows that $(i,p), (q,l) \notin V(\Pc)$, for any $p \leq l$ and $q \leq i$. Since, $\Pc$ is a polyomino, and hence connected, there must exist a path in $V(\Pc)$ from $(0,0)$ to $(i,j)$. However, any possible path in $V(\Pc)$ from $(0,0)$ to $(i,j)$ must either contain a vertex $(i,p)$ with  $p \leq l$ or a vertex  $ (q,l)$ with $q \leq i$, a contradiction. This yields $a\wedge b=(i,l) \in V(\Pc)$. A similar argument can be applied to conclude that $a \vee b=(k,j) \in V(\Pc)$. Moreover, the assertion that $\Pc$ is simple and planar as a distributive lattice, follows directly from the definition of polyominoes. 
\end{proof}

If a polyomino $\PP$ admits a structure of a distributive lattice on $V(\PP)$, then instead of $V(\PP)$, we refer to $\PP$ as a distributive lattice.
Let $(a,b) \in \NN\times \NN$. The edge $\{(a,b), (a+1,b)\}$ is called an {\em east step} and the edge $\{(a,b), (a,b+1)\}$ is called a {\em north step} in  $\NN\times \NN$. A sequence of vertices $\mathcal{S}:(a_0,b_0), (a_1,b_1) \ldots (a_k,b_k)$ in $\NN \times \NN$ is called a {\em north-east} path in $\NN \times \NN$, if $\{(a_i,b_i), (a_{i+1},b_{i+1})\}$ is either an east or a north step for each $i$. The vertices $(a_0,b_0)$ and $(a_k,b_k)$ are called the endpoints of $\mathcal{S}$.

Let $\mathcal{S}_1:(a_0,b_0), (a_1,b_1), \ldots , (a_k,b_k)$ and $\mathcal{S}_2:(c_0,d_0), (c_1,d_1), \ldots , (c_k,d_k)$ be two north-east paths in $\NN\times \NN$ such that $(a_0,b_0)=(c_0,d_0)$ and $(a_k,b_k)=(c_k,d_k)$.  If for all $1 \leq i,j \leq k-1$ we have $b_i > d_j$ whenever $a_i=c_j$, then  $\mathcal{S}_1$ is said to ``lie above" $\mathcal{S}_2$. The {\em parallelogram polyomino} $\Pc$ determined by $(\mathcal{S}_1,\mathcal{S}_2)$, where $\mathcal{S}_1$ lies above $\mathcal{S}_2$, is the region bounded above by $\mathcal{S}_1$ and bounded below by $\mathcal{S}_2$.  We refer to the path $S_1$ as the upper path of $\Pc$ and the path $\SS_2$ as the lower path of $\Pc$.  We will denote a parallelogram polyomino as $\Pc=(\SS_1, \SS_2)$ when we need to emphasize on its upper and lower paths.  In Figure \ref{fig:bigparall}, a parallelogram polyomino is shown. The thick line in Figure \ref{fig:bigparall} represents the upper path of $\Pc$ and the dashed line represents the lower path of $\Pc$. 
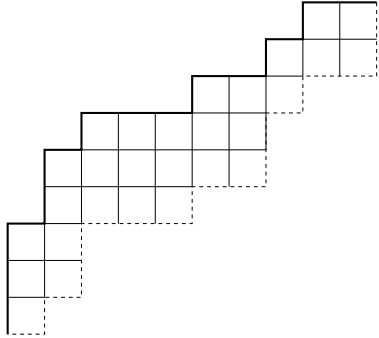
\begin{figure}[H]
  \centering
  \reflectbox{
  \rotatebox{270}{
    \resizebox{0.3\textwidth}{!}{
  \begin{tikzpicture}

\draw[dashed] (0,0)--(2,0)--(2,2)--(3,2)--(3,3)--(5,3)--(5,5)--(6,5)--(6,8)--(8,8)--(8,9)--(9,9)--(9,10);

\draw[ultra thick] (0,0)--(0,2)--(1,2)--(1,3)--(2,3)--(2,5)--(3,5)--(3,8)--(4,8)--(4,9)--(6,9)--(6,10)--(9,10);



\draw (0,1)--(2,1);\draw (1,2)--(2,2);\draw (2,3)--(4,3);\draw (2,4)--(5,4);\draw (3,5)--(5,5);\draw (3,6)--(6,6);\draw (3,7)--(6,7);\draw (4,8)--(6,8);\draw (6,9)--(8,9);

\draw (1,0)--(1,2);\draw (2,2)--(2,3);\draw (3,3)--(3,5);\draw (4,3)--(4,8);\draw (5,5)--(5,9);\draw (6,8)--(6,9);\draw (7,8)--(7,10); \draw(8,9)--(8,10);
   \end{tikzpicture}}}}
   \caption{A parallelogram polyomino}\label{fig:bigparall}
\end{figure}

A parallelogram polyomino that is also $L$-convex is known as a \emph{Ferrer diagram}. In \cite{EHQR} the authors prove that the coordinate ring of any  $L$-convex polyomino is isomorphic to the coordinate ring of a Ferrer diagram. 


 One can observe that every parallelogram polyomino $\Pc$ is a simple planar distributive lattice, as shown in the following proposition.
 
 \begin{Proposition}\label{prop:parallel=simplelattice}
A finite collection of cells $\Pc$ is a parallelogram polyomino if and only if $\Pc$ is a simple planar distributive lattice. 
 \end{Proposition}

 \begin{proof}
 
 Let $\Pc=(\SS_1, \SS_2)$ be a parallelogram polyomino. By a translation, we may assume that $\SS_1$ and $\SS_2$ meet at $(0,0)$ and $(m,n)$.  The definition of parallelogram polyomino together with Proposition~\ref{prop:convex=lattice} yields that $\Pc$ is a simple planar distributive lattice. 
 
 To show the converse, assume that $\Pc$ is a simple planar distributive lattice with bounding box $[(0,0), (m,n)]$. It follows from the definition of simple planar distributive lattice  that $\PP$ is convex polyomino. Note that $\rank(\Pc)=m+n$ as a lattice. Let $\mm_0: x_0<x_1<\cdots<x_{m+n-1}< x_{m+n}$ be the maximal chain of $\Pc$ with $x_t=(i_t,j_t)$ for all 
$0\leq t\leq m+n$ satisfying the following property: $(i_0,j_0)=(0,0)$, $(i_{m+n},j_{m+n})=(m,n)$, and for any $(k,\ell)\in V(\PP)$ with $k=i_t$ for some $t,$ if $l \geq j_t$ then $(k,l)=(i_s,j_s)$ for some $s\geq t$. We call such an $\mm_0$  the uppermost chain of $\Pc$. Similarly, one can define the lower most chain $\mm'_0$ of $\Pc$. Then it can be easily seen that $\Pc$ is a parallelogram polyomino determined by $(\mm_0, \mm'_0)$. 
 \end{proof}
 
 The following remark plays a vital role in subsequent text. 
 \begin{Remark}\label{rem:sameideal}
 Let $L$ be a simple planar distributive lattice and $a,b \in L$ be two incomparable elements in $L$. Let $c=a\vee b$ and $d=a \wedge b$. Then $a$ and $b$ determine an inner interval in $L$ with diagonal corners $c$ and $d$ and antidiagonal corners $a$ and $b$. Therefore, a typical generator $f_{ab}=x_ax_b-x_{a\vee b}x_{a\wedge b}= x_ax_b-x_cx_d$ of the join-meet ideal of $L$ is also an inner 2-minor of $L$. Similarly, any inner 2-minor of $L$ can be interpreted as a relation arising from two incomparable elements and their meet and join. This shows that the join-meet ideal and polyomino ideals of $L$ coincide.   
 \end{Remark}

\section{Hilbert series and rook complex of simple polyominoes}\label{sec:rookcomp}

In this section, we give a conjecture about the Hilbert series of the coordinate ring of simple polyominoes in terms of some rook arrangements on the cells of the parallelogram polyominoes. We first recall the definition of Hilbert function and Hilbert series (see also \cite{HHO,Vi}).

Let $R$ be a standard graded ring and $I$ be a homogeneous ideal. The \textit{Hilbert function} $\H_{R/I} : \mathbb{N} \rightarrow \mathbb{N}$ is defined by 
\[
\H_{R/I} (k) := \dim_K (R/I)_k
\]
where $(R/I)_k$ is the $k$-degree component of the gradation of $R/I$, while the \\ \emph{Hilbert-Poincar\'e series} of $R/I$ is
\[
\HS_{R/I} (t) := \sum_{k \in \NN} \H_{R/I}(k) t^k. 
\]
By the Hilbert-Serre theorem, the Hilbert-Poincar\'e series of $R/I$ is a rational function. In particular, by reducing this rational function we get
\[
\HS_{R/I}(t) = \frac{h(t)}{(1-t)^d}.
\]
for some $h(t) \in \mathbb{Z}[t]$, where $d$ is the Krull dimension of $R/I$. The degree of $\HS_{R/I}(t)$ as a rational function, namely $\deg h(t)-d$, is called \emph{a-invariant} of $R/I$, denoted by $a(R/I)$. It is known that whenever $R/I$ is Cohen-Macaulay we have $a(R/I)=\reg R/I-\depth R/I$, that is $\reg R/I=\deg h(t)$. In the latter $\reg R/I$ denotes the \emph{Castelnuovo-Mumford regularity} of $R/I$.

The well-known ``rook problem'' is the problem of enumerating the number of ways of placing $k$ non-attacking rooks on a pruned chessboard. Every simple polyomino $\PP$ can be viewed as a pruned chessboard. Given a simple polyomino $\PP$, recall that $k$ rooks placed on the cells of $\PP$ are said to be \emph{non-attacking} if they do not lie on the same row or the same column of cells of $\PP$, pairwise.  The maximum number of non-attacking rooks that can be placed on $\PP$, is called the \emph{rook number} of $\PP$, denoted by $r(\PP)$. By abuse of notation, one identifies the rooks that can be placed on $\PP$ with the cells of $\PP$. We observe that for any cell $C \in \PP$, the set $\{C\}$ is a $1$ non-attacking rook. Moreover, for any set of non-attacking rooks $F \subset \PP$, the subset $G\subset F$ is also a set of non-attacking rooks. This yields that the set $\RR$ of sets of non-attacking rooks is a simplicial complex and 
\[
\mathcal{R}=\RR_0 \cup \RR_1 \cup \ldots \cup \RR_{r(\PP)},
\]
where for any $k=0,\ldots, r(\PP)$, $\RR_k$ contains the sets of $k$ non-attacking rooks, with $\RR_0=\varnothing$. Set $r_k = |\RR_k|$. The polynomial 
\[
r_{\PP}(t)=\sum_{k=0}^{r(\PP)} r_k t^k
\] 
is called the {\em rook polynomial} of $\PP$. The rook polynomials are widely studied in combinatorics. We refer to \cite{GG}, for more information on this topic. 
 
Next, we introduce an equivalence relation on the set $\RR_k$ for $2 \leq k \leq r(\PP)$. For this aim, we define the following.

\begin{Definition}\label{def:equivrook}
Two non-attacking rooks $R_1$ and $R_2$ of $\PP$ are \emph{switching rooks} if they are diagonal (resp. antidiagonal) cells of a rectangle of $\PP$.  Let $R_1'$ and $R_2'$ be the antidiagonal (resp. diagonal).
Observe that if $F\in \RR$  and $R_1,R_2 \in F$ are switching rooks, then the set $F'\setminus\{R_1,R_2\}\cup \{R_1',R_2'\}\in \RR$. The replacement of $R_1$ and $R_2$ by $R_1'$ and $R_2'$ is called \emph{switch} of $R_1$ and $R_2$. \\
There exists a natural equivalence relation $\sim$ on $\RR_k$ given as:  $F_1,F_2 \in \RR_k$  are \emph{equivalent} if one can obtain $F_2$ from $F_1$ after some switches. We define the quotient set
\[
\widetilde{\RR}_k=\RR_k/\sim.
\] 
We observe that the rook number $r(\PP)$ does not change. We define the polynomial
\[
\widetilde{r}_{\PP}(t)=\sum\limits_{k=0}^{r(\PP)} |\widetilde{\RR}_k|t^k.
\]
\end{Definition}
With the notation introduced above, we state the following:
\begin{Conjecture}\label{conj:rtilde}
Let $\PP$ a simple polyomino. Then $h(t)=\widetilde{r}_{\PP}(t)$.
\end{Conjecture}

The following example depicts the construction of a rook complex $\mathcal{R}$ of a polyomino and the quotient set $\widetilde{\RR}:=\RR / \sim $. 
\begin{Example}
We describe $\RR$ and $\widetilde{\RR}$ for the simple polyomino $\PP$ in Figure \ref{fig:rookcom}. The polyomino $\PP$ consists of seven cells labelled as $A,B,C,D,E,F,G$ and  $r(\PP)=3$. The rook complex $\RR=\RR_0 \cup \RR_1 \cup \RR_2 \cup \RR_3$ of $\PP$ is given below.
\begin{itemize}
    \item[$\RR_0=$] \{$\varnothing$\}
    \item[$\RR_1=$]\{ $\{A\}$, $\{B\}$, $\{C\}$, $\{D\}$, $\{E\}$, $\{F\}$, $\{G\}$\}
    \item[$\RR_2=$] \{$\{A,D\}$, $\{A,E\}$, $\{A,G\}$, $\{B, C\}$, $\{B, E\}$, $\{B,F \}$, $\{B, G\}$, $\{C,G\}$, $\{D,F\}$, $\{D,G\}$, $\{E,F\}$, $\{F,G\}$\}
    \item[$\RR_3=$] \{$\{A,D,G\}$, $\{B, C, G\}$, $\{B, E, F\}$, $\{B,F, G\}$, $\{D,F,G\}$\}.
\end{itemize}
This gives,
\[
r_{\PP}(t)=1+7t+12t^2+5t^3
\]

We observe that $A$ and $D$ are switching rooks and they can be switched with $B$ and $C$. Then
\[
\{A,D\} \sim \{B,C\}, \ \{A,D,G\}\sim \{B,C,G\}
\]
and 
\[
\widetilde{r}_{\PP}(t)=1+7t+11t^2+4t^3.
\]
By using \texttt{Macaulay2}, one can see that $h(t)=\widetilde{r}_{\PP}(t)$.
\begin{figure}[H]
  \centering
    \resizebox{0.3\textwidth}{!}{
  \begin{tikzpicture}
 \draw (0,0)--(2,0);
 \draw (0,1)--(3,1);
 \draw (0,2)--(3,2);
 \draw (0,3)--(1,3);
 \draw (2,3)--(3,3);
 
\draw (0,0)--(0,3);
\draw (1,0)--(1,3);
\draw (2,0)--(2,3);
\draw (3,1)--(3,3);

\node at (0.5,0.5){$A$};
\node at (1.5,0.5){$B$};
\node at (0.5,1.5){$C$};
\node at (1.5,1.5){$D$};
\node at (2.5,1.5){$E$};
\node at (0.5,2.5){$F$};
\node at (2.5,2.5){$G$};
   \end{tikzpicture}}
   \caption{A simple polyomino}\label{fig:rookcom}
\end{figure}
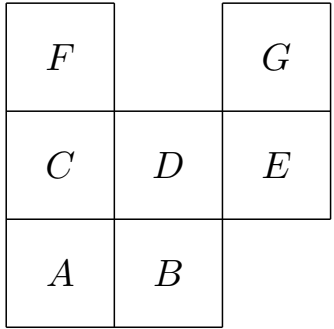

\end{Example}

 As proved in \cite{RR}, Conjecture \ref{conj:rtilde} holds true for the class of simple thin polyominoes. In fact, by definition of a thin polyomino, it does not contain a square tetromino (see \cite{RR}) as a subpolyomino. Therefore, a simple thin polyomino $\PP$ does not contain switching rooks and $\widetilde{r}_{\PP}(t)= r_{\PP}(t)=h(t)$. Moreover, by computational approach we obtain the following 
 \begin{Theorem}
 Let $\PP$ be a simple polyomino with $\rank \PP \leq 11$. Then $h(t)=\widetilde{r}_{\PP}(t)$.
 \end{Theorem}
\begin{proof}
To prove the claim we have implemented a computer program that, for a fixed
number $n$, performs the following steps:
\begin{itemize}
    \item[(S1)] compute the set of all the simple polyominoes of rank $n$;
    \item[(S2)] for any polyomino in (S1) compute the polynomial $h(t)$;
    \item[(S3)] for any polyomino in (S1) compute the polynomial $\widetilde{r}_\PP(t)$;
    \item[(S4)] check whether the polynomial from step (S2) is equal to the polynomial from step (S3).
\end{itemize}

In particular, for step (S1) we slightly modified the implementation given in \cite{MRRsite}. For step (S2) we used the {\tt Macaulay2} functions for the Hilbert series. For step (S3), we constructed the rook complex $\RR$ as the independence complex of the graph $G$ with $V(G)=\{C : C \in \PP \}$ and 
\[
E(G)=\{\{C,D\} : \mbox{the cells $C$ and $D$ lie on the same row or column}\},
\]
and then, by introducing the equivalence relation, we constructed $\widetilde{\RR}$. Finally we refer the reader to \cite{QRR} for a complete description of the algorithm that we used.
\end{proof}

Now we will prove Conjecture \ref{conj:rtilde} for parallelogram polyominoes.

\begin{Theorem}\label{theo:hilbparall}
Let $\PP$ a parallelogram polyomino. Then $h(t)=\widetilde{r}_{\PP}(t)$.
\end{Theorem}

From Proposition \ref{prop:parallel=simplelattice}, we know that a parallelogram polyomino can be seen as a simple planar distributive lattice. Furthermore, Remark~\ref{rem:sameideal} shows that the join-meet ideal and polyomino ideal of a simple planar distributive lattice coincides. To achieve our aim, we will first recall some notions related to simple planar distributive lattices and their Hilbert series. 

Let $x, y \in L$ such that $y$ covers $x$, that is, $x <y$ and there is no $z \in L$ such that $x<z<y$. Then the edge between $x$ and $y$ in the Hasse diagram of $L$ can be represented by $x \rightarrow y$. Recall from \cite{BGS} that an edge-labeling $\lambda$ of L is an integer labelling of the edges in Hasse diagram of $L$. Each chain in $L$, say $\mathfrak{c}: x_0 \rightarrow x_1 \rightarrow x_2 \rightarrow \ldots \rightarrow x_k$  can be labelled by a $k$-tuple $\lambda(\mathfrak{c})=(\lambda(x_0 \rightarrow x_1), \lambda(x_1 \rightarrow x_2),\ldots, \lambda(x_{k-1} \rightarrow x_k) )$. One can compare two such $k$-tuples $(a_1, \ldots , a_k)$ and $(b_1, \ldots , b_k)$ lexicographically, that is, $(a_1, \ldots , a_k) <_{\lex}(b_1, \ldots , b_k)$,  if the most-left nonzero component of the vector $(a_1-b_1, \ldots , a_k-b_k)$ is positive.

\begin{Definition}\cite[Definition 2.1]{BGS}
An edge labelling $\lambda$ of $L$ is called EL-labelling if for every interval $[x,y]$ in $L$, $\lambda$ satisfies the following:
\begin{enumerate}
\item there  is a unique chain $\mathfrak{c} : x=x_0 \rightarrow x_1 \rightarrow x_2 \rightarrow \ldots \rightarrow x_k=y$ such hat $\lambda(x_0 \rightarrow x_1) \leq \lambda(x_1 \rightarrow x_2) \leq \ldots \leq \lambda(x_{k-1} \rightarrow x_k) )$.

\item for every other chain $\mathfrak{b} : x=y_0 \rightarrow y_1 \rightarrow y_2 \rightarrow \ldots \rightarrow y_k=y$, we have $\lambda (\mathfrak{b}) >_{\lex} \lambda (\mathfrak{c})$.
\end{enumerate}

\end{Definition}

\noindent
In Figure~\ref{fig:ELlab}, we give an illustration of  {\em EL}-labeling $\lambda$.

Let $\rank(L)=d+1$. Then for each maximal chain $\mathfrak{m}: \min{L}=x_0 \rightarrow x_1 \rightarrow x_2 \rightarrow \ldots \rightarrow x_{d+1}=\max(L)$, the {\em descent set} of $\mathfrak{m}$ is $D(\mathfrak{m})= \{i : \lambda(x_{i-1} \rightarrow x_i)  > \lambda(x_{i} \rightarrow x_{i+1}) \}$. Then by following \cite[Theorem 2.2]{BGS}, for any $S \subset [d]$, we set $\beta(S)$ to be the number of maximal chains $\mathfrak{m}$ in $L$ such that $D(\mathfrak{m})=S$. It is known from \cite{BGS}, that 
\[
\HS_{K[L]}(t)=\frac{h(t)}{(1-t)^{d+2}}
\]
where
 \[
 h(t)=\sum_{S \subset [d]} \beta(S) t^{|S|}.
 \]

Our main goal is to interpret $ h(t)$ in terms of $\widetilde{r}_{\PP}(t)$.

In the following, we recall the definition of uppermost chain (already used in the proof of Proposition \ref{prop:parallel=simplelattice}), adding a nice EL-labelling.

\begin{Definition}\label{def:EL} 
Let $L$ be a simple planar distributive lattice of rank $d+1$.
\begin{enumerate}
    \item  Let $\mm_0: x_0<x_1<\cdots<x_d< x_{d+1}$ be the maximal chain of $L$ with $x_t=(i_t,j_t)$ for all 
$0\leq t\leq d+1$ satisfying the following property: $(i_0,j_0)=(0,0)$, $(i_{d+1},j_{d+1})=\max L$, and for any $(k,\ell)\in L$ with $k=i_t$ for some $t,$ if $\ell \geq j_t$ then $(k,\ell)=(i_s,j_s)$ for some $s\geq t$. We call such an $\mm_0$  the \textit{uppermost chain} of $L.$ We label the edges of $\mm_0$ by $\lambda(x_t\to x_{t+1})=t+1$ for $0\leq t\leq d.$
    \item Let $x, y \in L$ such that $x < y$. Then the {\em uppermost chain} from $x$ to $y$ in $L$ is the uppermost chain of the sublattice $L\cap [x,y]$.
\end{enumerate}
\end{Definition}

Figure~\ref{uppermostchains}.(I) illustrates an example of an uppermost chain between two elements $x,y$ of $L$, while Figure~\ref{uppermostchains}.(II) illustrates an example of an uppermost chain of a lattice $L$. The uppermost chains are indicated by thick lines. 

\begin{figure}[H]
\centering
\begin{subfigure}{0.5\textwidth}
\centering
\begin{tikzpicture}
\draw (1,0)--(0,1);
\draw (1,0)--(2,1);
\draw[very thick] (2,1)--(1,2);
\draw (1,2)--(0,1);
\draw (2,1)--(3,2);
\draw (2,3)--(3,2);
\draw[very thick] (1,2)--(2,3);
\draw (2,3)--(3,2);
\draw (1,2)--(0,3);
\draw (0,3)--(1,4);
\draw (2,3)--(1,4);
\filldraw (2,1) circle (\rad) node [anchor=west]{$x$};
\filldraw (2,3) circle (\rad) node [anchor=west]{$y$};
\end{tikzpicture}
\caption{The uppermost chain between $x$ and $y$.}
\end{subfigure}%
\begin{subfigure}{0.5\textwidth}
\centering
\begin{tikzpicture}
\draw[very thick] (1,0)--(0,1);
\draw (1,0)--(2,1);
\draw (2,1)--(1,2);
\draw[very thick] (1,2)--(0,1);
\draw (2,1)--(3,2);
\draw (2,3)--(3,2);
\draw(1,2)--(2,3);
\draw (2,3)--(3,2);
\draw[very thick] (1,2)--(0,3);
\draw[very thick] (0,3)--(1,4);
\draw (2,3)--(1,4);
\filldraw (1,0) circle (\rad);
\filldraw (1,4) circle (\rad);
\end{tikzpicture}
\caption{The uppermost chain of $L$}
\end{subfigure}
\caption{Two examples of uppermost chains}\label{uppermostchains}
\end{figure}

In \cite{EQR}, the following {\em EL}-labelling is defined for simple planar distributive lattices. 

\begin{Definition}\label{def:HasseLabel}
We label all the edges in the Hasse diagram of $L$ as follows. If $i_{t+1}=i_t+1$, in other words if $x_t\to x_{t+1}$, is a horizontal edge, then we label by $t+1$ all the edges of $L$ of the form $(i_t,j)\to (i_{t+1},j)$. If $j_{t+1}=j_t+1$, that is, if $x_t\to x_{t+1}$ is a vertical edge, then we label by $t+1$ all the edges of $L$ of the form $(i,j_t)\to (i,j_{t+1})$. In \cite[Proposition 6]{EQR}, it is shown that $\lambda$ is an {\em EL}-labelling. 
\end{Definition}

In Figure~\ref{fig:ELlab}, we use Definition \ref{def:HasseLabel} for the   {\em EL}-labeling $\lambda$. The chain marked with thick line is the uppermost chain of $L$.

\begin{figure}[H]
  \centering
    \resizebox{0.4\textwidth}{!}{
  \begin{tikzpicture}
 
\draw[thick] (0,0) --  (3,0);
\draw[thick] (0,1) --  (4,1);
\draw[thick] (1,2) --  (4,2);
\draw[thick] (2,3) --  (4,3);

\draw[thick] (0,0) --  (0,1);
\draw[thick] (1,0) --  (1,2);
\draw[thick] (2,0) --  (2,3);
\draw[thick] (3,0) --  (3,3);
\draw[thick] (4,1) --  (4,3);

\node at (0,0.5)[anchor=east]{\bf\scriptsize 1};
\node at (0.3,1)[anchor=south]{\bf\scriptsize 2};
\node at (1,1.5)[anchor=east]{\bf\scriptsize 3};
\node at (1.3,2)[anchor=south]{\bf\scriptsize 4};
\node at (2,2.5)[anchor=east]{\bf\scriptsize 5};
\node at (2.3,3)[anchor=south]{\bf\scriptsize 6};
\node at (3.3,3)[anchor=south]{\bf\scriptsize 7};

\node at (1,0.5)[anchor=east]{\scriptsize 1};
\node at (2,0.5)[anchor=east]{\scriptsize 1};
\node at (3,0.5)[anchor=east]{\scriptsize 1};
\node at (0.3,0)[anchor=south]{\scriptsize 2};
\node at (2,1.5)[anchor=east]{\scriptsize 3};
\node at (3,1.5)[anchor=east]{\scriptsize 3};
\node at (4,1.5)[anchor=east]{\scriptsize 3};
\node at (1.3,1)[anchor=south]{\scriptsize 4};
\node at (1.3,0)[anchor=south]{\scriptsize 4};
\node at (3,2.5)[anchor=east]{\scriptsize 5};
\node at (4,2.5)[anchor=east]{\scriptsize 5};
\node at (2.3,2)[anchor=south]{\scriptsize 6};
\node at (2.3,1)[anchor=south]{\scriptsize 6};
\node at (2.3,0)[anchor=south]{\scriptsize 6};
\node at (3.3,2)[anchor=south]{\scriptsize 7};
\node at (3.3,1)[anchor=south]{\scriptsize 7};
   \end{tikzpicture}}
   \caption{The $EL$-labelling for a parallelogram polyomino}\label{fig:ELlab}
\end{figure}
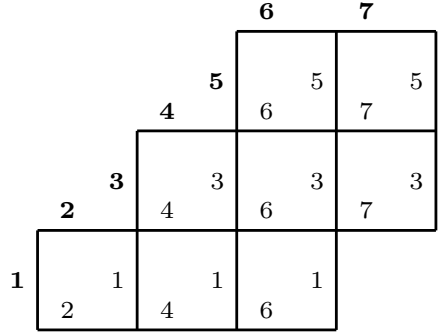


Throughout the following text, we will follow the {\em EL}-labelling given in Definition~\ref{def:EL}. The following remarks are immediate consequences of Definition~\ref{def:EL} (see Figure \ref{fig:ELlab}).

\begin{Remark}\label{rem:labelling}
\begin{enumerate}
\item Let $(i,j) \rightarrow (i,j+1)$ and $(k,l) \rightarrow (k,l+1)$ be two edges in $L$ with $i \leq k$ and $j+1 \leq l$. Then $\lambda ((i,j) \rightarrow (i,j+1)) < \lambda ((k,l) \rightarrow (k,l+1) )$.

\item  Let $(i,j) \rightarrow (i+1,j)$ and $(k,l) \rightarrow (k+1,l)$ be two edges in $L$ with $i+1 \leq k$ and $j \leq l$. Then $\lambda ((i,j) \rightarrow (i+1,j)) < \lambda ((k,l) \rightarrow (k+1,l) )$. 

\item   Let $(i,j) \rightarrow (i,j+1)$ and $(k,l) \rightarrow (k+1,l)$ be two edges in $L$ with $i \leq k$ and $j+1 \leq l$. Then $\lambda ((i,j) \rightarrow (i,j+1)) < \lambda ((k,l) \rightarrow (k+1,l) )$.

\item Let $(i,j) \rightarrow (i+1,j)$ and $(i+1,j) \rightarrow (i+1,j+1)$ be two edges in $L$. If $(i,j+1) \notin L$, then $(i,j) \rightarrow (i+1,j)$ and $(i+1,j) \rightarrow (i+1,j+1)$ appear in the uppermost chain of $L$ and  $\lambda ((i,j) \rightarrow (i+1,j)) < \lambda ((i+1,j) \rightarrow (i+1,j+1))$. However, if $(i,j+1) \in L$, then  $(i,j) \rightarrow (i,j+1)$ and $(i,j+1) \rightarrow (i+1,j+1)$  are edges in $L$. Moreover, due to (3) we have $\lambda ((i,j) \rightarrow (i,j+1)) < \lambda ((i,j+1) \rightarrow (i+1,j+1) )$. Following the Definition~\ref{def:EL}, we have 
\[
\lambda ((i,j) \rightarrow (i,j+1)) = \lambda ((i+1,j) \rightarrow (i+1,j+1))
\]
 and 
\[
  \lambda ((i,j+1) \rightarrow (i+1,j+1) ) =  \lambda ((i,j) \rightarrow (i+1,j) )
  \]
  which gives 
  \[
  \lambda((i,j) \rightarrow (i+1,j)) > \lambda ((i+1,j) \rightarrow (i+1,j+1))
  \]
 
  \item From (1)--(4), we can compute the descent set of a maximal chain $\mathfrak{m}$ in $L$. If $\mathfrak{m}$ contains edges of the form  $(i,j) \rightarrow (i+1,j)$ and $(i+1,j) \rightarrow (i+1,j+1)$ and $(i,j+1) \in L$, then we have a descent at $(i+1,j)$. 

  \item Let $x=(i,j), y=(p,q) \in L$ with $i<p$ and $j<q$ and let $\mathfrak c : x=x_0 < x_1 <\ldots < x_l=y$ be the uppermost chain between $x$ and $y$. It follows that if $(i,j+1)\in L$, then $x_1=(i,j+1)$. That is, in $\mathfrak{c}$ there are no descents. Similarly one proves that in an uppermost chain there are no descents. 
 \end{enumerate}
\end{Remark}


The following definition is needed for the Proposition \ref{prop:descent}.
\begin{Definition}\label{def:cell}
Let $C=[(i,j), (i+1,j+1)]$ be a cell in a simple planar distributive lattice $L$. Then the lower left corner $(i,j)$ of $C$ is denoted by $l(C)$. Given any maximal chain $\mathfrak{m}$ in a simple planar distributive lattice $L$, we say that $\mathfrak{m}$  has a descent at cell $C$ if $\mathfrak{m}$ passes through the edges $(i,j) \rightarrow (i+1,j)$ and $(i+1,j) \rightarrow (i+1,j+1)$. 
\end{Definition}

\begin{Proposition}\label{prop:descent}
Let $L$ be a simple planar distributive lattice. Then the following are equivalent.

\begin{enumerate}
    \item There exists a maximal chain $\mathfrak{m}$ in $L$ with $|D(\mm)|=r$. 
    \item There exists $C_1$, $C_2$, $\ldots$, $C_r$ cells of $L$  with $l(C_k)=(i_k,j_k)$ for $ 1 \leq k \leq r$
\[
i_1< i_2 < \ldots < i_r \mbox{ and } j_1< j_2 < \ldots < j_r. 
\]
\end{enumerate}
Observe that the chain $\mm$ has descents exactly at $C_1,\ldots C_r$.
\end{Proposition}
\begin{proof}
$(1)\Rightarrow (2)$ Let $\mm: x_0<x_1<\cdots<x_d< x_{d+1}$  be a maximal chain with descent set $D(\mathfrak m)=\{l_1,\ldots,l_r\}$ with $l_1 < \ldots < l_r$. Then $x_{l_1}< x_{l_2}< \ldots <x_{l_r}$. From Remark \ref{rem:labelling}.(5) and Definition \ref{def:cell}, for any $i \in \{1,\ldots, r\}$ there exists a cell $C_i$ such that lower right corner and the lower left corner of $C_i$ are $x_{l_i}$  and $x_{l_i-1}$, respectively. 
We now prove that for any $k=1,\ldots,r-1$  we have $i_k<i_{k+1}$ and $j_{k}<j_{k+1}$. From the fact that, $x_{l_{k}-1}<x_{l_{k+1}-1}$ it follows that $i_k \leq i_{k+1}$. By contraposition, assume that $i_k = i_{k+1}$ for some $k$, then we have that $x_{l_k-1}=(i_k,j_k)$, $x_{l_k}=(i_k+1,j_k)$ and $x_{l_{k+1}-1}=(i_k,j_{k+1})$. That is, $x_{l_k} \not < x_{l_{k+1}-1}$, hence $\mm$ is not a chain and this is a contradiction. Hence $i_k < i_{k+1}$ and similarly $j_k < j_{k+1}$. 

$(2)\Rightarrow (1)$  For a cell $C$ in $L$, let $r(C)$ be its lower right corner. Let $\mathfrak c_0$ be the uppermost chain between $(0,0)$ and $(i_1,j_1)$ and for $1\leq k \leq r-1$ let $\mathfrak{c}_{k}$ be the uppermost chain between $r(C_k)=(i_k+1,j_k)$ and $l(C_{k+1})=(i_{k+1},j_{k+1})$ and let $\mathfrak c_r$ be the uppermost chain between $r(C_{r})$ and $(m,n)$. From the concatenation of $\mathfrak c_0 \mathfrak c_1 \ldots \mathfrak c_r $, we obtain in a natural way a maximal chain $\mathfrak m$ of $L$. We prove that $\mathfrak{m}$ has descent at  $C_1, C_2, \ldots, C_r$. We fix $k\in \{1,\ldots,r\}$. Since  $(i_k,j_k)\rightarrow (i_k+1,j_k) \in E(\mathfrak m)$, it is sufficient to prove $(i_k+1,j_k)\rightarrow (i_k+1,j_k+1) \in E(\mathfrak m)$. The assertion follows from the inequalities on $i_1< \ldots < i_r$ and $j_1<\ldots < j_r$ and Remark \ref{rem:labelling}.(6) applied to the uppermost chain $\mathfrak c_{k+1}$. Therefore, $\mathfrak m$ has descent at $C_k$. This completes the proof.
\end{proof}

In order to prove Theorem \ref{theo:hilbparall}, we premise the following lemma which shows that given any set of non-attacking rooks in a parallelogram polyomino, one can find an equivalent set of non-attacking rooks whose lower left corners appear in a chain.

\begin{Lemma}\label{lem:ordrooks}
Let $\PP$ be a parallelogram polyomino and let $F=\{A_1,\ldots, A_d\}\in \RR$. Then there exists  $G=\{B_1,\ldots, B_d\}\in \RR$  with $l(B_k)=(i_k,j_k)$ for $ 1 \leq k \leq d$ such that 
\[
i_1< i_2 < \ldots < i_d \mbox{ and } j_1< j_2 < \ldots < j_d 
\]
and $F\sim G$.
\end{Lemma}
\begin{proof}
Let $l(A_i)=(x_i,y_i)$ for $i =1,\ldots, d$.
We prove the assertion by applying induction on $d$.

Let $d=2$ and assume that $\mathcal{A}=\{A_1,A_2\}$ is labelled such that $x_1 < x_2$. If $y_1 < y_2$ then the statement holds trivially. If $y_1>y_2$, then by using the assumption that $\PP$ is a parallelogram polyomino and hence a distributive lattice, we conclude that the join $b_1=(x_2,y_1)$ and the meet $b_2=(x_1,y_2)$ of $\mathcal{A}_1$ and $\mathcal{A}_2$ belong to $V(\Pc)$. In particular, $b_1$ and $b_2$ are lower left corners of some cells of $\PP$ and $A_1$ and $A_2$ are antidiagonal cells of a rectangle of $\PP$. Let $B_1$ and $B_2$ be the cells having lower left corners $b_1$ and $b_2$, respectively. It follows that the set $\{B_1,B_2\}$ satisfies the assertion. 
Now, let $d>2$ and assume that the assertion is true for all of the sets containing $d-1$ non-attacking rooks.
We label the elements of $\mathcal{A}$ in a way such that $x_1< x_2< \ldots < x_d$. Let $k\in \{1,\ldots, d\}$ such that $y_k < y_{i}$ for any $i\neq k$. If $k=1$, then we set $B_1=A_1$ and we apply the inductive hypothesis on the set $\{A_2,\ldots , A_d\}$ to get the desired result.\\
If $k>1$, then let $B_1$ and $C_k$ be the cells whose lower left corners are respectively $(x_1,y_k)$ and $(x_k,y_1)$. Then $\{A_2,\ldots, A_{k-1},C_k,A_{k+1},\ldots, A_d\}$ is a set of $d-1$ non-attacking rooks and by applying the inductive hypothesis the assertion follows.
 \end{proof}
Now we state the proof of Theorem \ref{theo:hilbparall}.

\begin{proof}[Proof of Theorem \ref{theo:hilbparall}]
Let 
\[
\HS_{K[\PP]}(t)=\frac{\sum\limits_{k} h_{k}t^k}{(1-t)^{\dim K[\PP]}}.
\]
We show that for any $k$ one has $\widetilde{r}_k=h_k$. For $k=0,1$ one has $\widetilde{r}_k=h_k$.\\
For $k\geq 2$, by Proposition \ref{prop:descent}  the maximal chains with descent set of cardinality $k$ in $\PP$ seen as a planar distributive lattice are in bijection with the sets $F$ of non-attacking rooks $B_1,\ldots, B_k$ with $l(B_\ell)=(i_\ell,j_\ell)$ for $ 1 \leq \ell \leq k$ such that $i_1< i_2 < \ldots < i_k \mbox{ and } j_1< j_2 < \ldots < j_k$. Thanks to Lemma \ref{lem:ordrooks}, such sets $F$ are the representatives of the equivalence classes of $\RR_k / \sim$, that is $\widetilde{r}_k=h_k$.
\end{proof}

As a consequence of Theorem \ref{theo:hilbparall}, we observe that the Conjecture~\ref{conj:rtilde} holds for $L$-convex polyominoes, too. As stated in Section~\ref{sec:pre}, the coordinate ring of an $L$-convex polyomino is isomorphic to the coordinate ring of a suitable Ferrer diagram. Then the conclusion follows from the fact that every Ferrer diagram is a particular parallelogram polyomino. We also note that the Hilbert series of Ferrer diagram was given in \cite{CN}.

From Lemma \ref{lem:dim}, it follows that for a parallelogram polyomino $\PP$ the coordinate ring $K[\PP]$ is a Cohen-Macaulay domain. Furthermore, as mentioned at the beginning of this section, $\deg h(t)=\reg K[\PP]$. Hence we obtain the following corollary of Theorem \ref{theo:hilbparall}.
 \begin{Corollary}
 Let $\PP$ be a parallelogram polyomino. Then $\reg K[\PP] = r(\PP)$.
 \end{Corollary}

\section{Gorenstein parallelogram polyominoes}\label{sec:gor}
Given a polyomino $\PP$, we call $\PP$ Gorenstein if $K[\PP]$ is Gorenstein. In this section we discuss the Gorenstein parallelogram polyominoes. Although the Gorenstein distributive lattices are completely characterized in \cite{H},  we plan to give a combinatorial interpretation of the Gorenstein parallelogram polyominoes in the language of polyominoes. Our aim is to compare the conditions on a parallelogram polyomino to be Gorenstein with the conditions found in \cite{EHQR} for $L$-convex polyominoes and in \cite{RR} for simple thin polyominoes. 

Let $\MM(\PP)$ be the set of the maximal rectangles of $\PP$. We generalize \cite[Definition 4.1]{RR} with the following.

\begin{Definition}
Let $S$ be a rectangular (resp. square) subpolyomino of a parallelogram polyomino $\PP$. Then $S$ is said to be \emph{single} if there exists a unique maximal rectangle $R \in \MM(\PP)$ such that $S \subseteq R$ and $S \cap R' =\varnothing$ for all $R' \in \MM(\PP)$ with $R'\neq R$. We say that $\PP$ has the \emph{$S$-property} if each maximal rectangle $R$ of $\PP$ has a unique single square.
\end{Definition}

To see an illustration of the above definition, consider the parallelogram polyomino $\PP$ given in Figure~\ref{fig:shornotwell}.(\textsc{\romannumeral 1}). $\Pc$ has six maximal rectangles
\[
\{A,B\},\{B,C,E\}, \{C,D,E,F,\}, \{D,F,H\}, \{E,F,G\}, \{F,G,H,I\}
\]
The maximal rectangle $\{A,B\}$ has $A$ as its single square, and the maximal rectangle $\{F,G,H,I\}$ has $I$ as its single square. However, other rectangles do not have a single square or a single rectangle because each of their cells belong to other rectangles as well. The maximal rectangles $\{A,B\}$ and $\{F,G,H,I\}$ are special in a sense that $\{A,B\}$  is the unique maximal rectangle containing $\min (\PP)$ (as a distributive lattice) and  $\{F,G,H,I\}$ is the unique maximal rectangle containing $\max (\PP)$.

Next, we prove that if a maximal rectangle $R$ in a parallelogram polyomino $\PP$ contains either $\min(\PP)$ or $\max(\PP)$ then $R$ must contain a single rectangle. Given a parallelogram polyomino $\PP$, we set $\min (\PP)=(0,0)$ throughout the following text. 

\begin{Lemma}\label{lem:uniquerec}
Let $\PP$ be a parallelogram polyomino. Then there exists a unique $R \in \MM(\PP)$ such that $(0,0) \in V(R)$. In particular, the maximal rectangle $R$ has a single rectangle.
\end{Lemma}
\begin{proof}
By contraposition, assume that there are two distinct maximal rectangles $R,R'$ of $\PP$, such that $(0,0)\in V(R)\cap V(R')$. Let $a,b,c,d \in V(\PP)$ be such that $V(R)=[(0,0),(a,b)]$ and $V(R')=[(0,0),(c,d)]$. Since $R$ and $R'$ are distinct, without loss of generality, we may assume that $a <c$ and $b>d$. From Proposition~\ref{prop:parallel=simplelattice}, it follows that $\PP$ is a simple planar distributive lattice. Therefore, $(c,b)\in V(\PP)$ because it is the join of $(a,b)$ and $(c,d)$. This shows that the rectangle $\widetilde{R}$ with $V(\widetilde{R})=[(0,0),(c,b)]$ contains both $R$ and $R'$, a contradiction to the maximality of $R$ and $R'$. Therefore, we conclude that there exists a unique maximal rectangle $R$ that contains $(0,0)$. In addition, we obtain that the cell with lower left corner $(0,0)$ only belongs to $R$. This shows that $R$ must have a single rectangle.
\end{proof}

In the following text, for a given parallelogram polyomino $\PP$, the unique maximal rectangle of $\PP$ containing  $\min (\PP) =(0,0)$ is denoted by $R_0$. Let $\PP'$ be a subpolyomino of $\PP$. Then $\PP\setminus \PP'$ is a collection of cells obtained by removing all cells of $\PP'$ from $\PP$. Next, we introduce a new family of parallelogram polyominoes.

\begin{Definition}
A parallelogram polyomino $\PP$ is said to be \emph{shortenable} if $\PP \setminus R_0$ is a parallelogram polyomino.  Moreover, $\PP$ is \emph{well-shortenable} if $\PP$ is shortenable and either $\PP\setminus R_0$ is a rectangle or $\PP\setminus R_0$ is a well-shortenable parallelogram polyomino. The sequence of polyominoes $\{\PP_i\}_{i=1,\ldots, l}$ such that $\PP_1=\PP\setminus R_0$, and $\PP_{i+1}=\PP_i\setminus R_i$  where $R_i$ is the unique rectangle containing $\min (\PP_i)$, is called the \emph{derived sequence} of $\PP$.
\end{Definition}

We observe that a thin parallelogram polyomino and an $L$-convex parallelogram polyomino (Ferrer diagram) are well-shortenable.
In particular, for a Ferrer diagram the definition of  derived sequence coincides with the one of \cite{EHQR}.
\begin{Example}
We give an example of a shortenable polyomino that is not well-shortenable. Let $\PP$ be the parallelogram polyomino in Figure \ref{fig:shornotwell}.(\textsc{\romannumeral 1}). We observe that the maximal rectangle $R_0$ of $\PP$ is the maximal rectangle on the cells $A$ and $B$, and the polyomino $\PP_1= \PP\setminus R_0$ is a parallelogram polyomino (see Figure \ref{fig:shornotwell}.(\textsc{\romannumeral 2})). Then $\PP$ is shortenable. However, the rectangle $R_1$ on the cells $\{C,D,E,F\}$ in $\PP_1$ is such that $\PP_1\setminus R_1$ is not a parallelogram polyomino (without rotation), see Figure \ref{fig:shornotwell}.(\textsc{\romannumeral 3}).
 
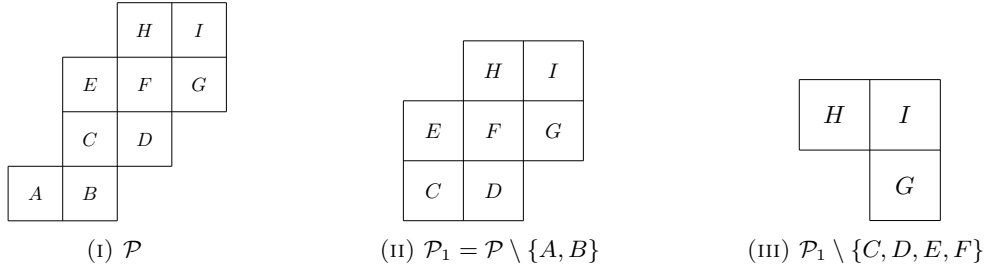
\begin{figure}[H]
  \centering
  \begin{subfigure}[t]{0.333 \textwidth}
  \centering
    \resizebox{0.6\textwidth}{!}{
  \begin{tikzpicture}
 \draw (0,0)--(2,0);
 \draw (0,1)--(3,1);
 \draw (1,2)--(4,2);
 \draw (1,3)--(4,3); 
 \draw (2,4)--(4,4);

 \draw (0,0)--(0,1);
 \draw (1,0)--(1,3);
 \draw (2,0)--(2,4);
 \draw (3,1)--(3,4);
 \draw (4,2)--(4,4);

\node at (0.5,0.5){$A$};
\node at (1.5,0.5){$B$};
\node at (1.5,1.5){$C$};
\node at (2.5,1.5){$D$};
\node at (1.5,2.5){$E$};
\node at (2.5,2.5){$F$};
\node at (3.5,2.5){$G$};
\node at (2.5,3.5){$H$};
\node at (3.5,3.5){$I$};
   \end{tikzpicture}}
   \caption{$\PP$}
   \end{subfigure}%
   \begin{subfigure}[t]{0.333 \textwidth}
   \centering
    \resizebox{0.5\textwidth}{!}{
  \begin{tikzpicture}
 \draw (1,1)--(3,1);
 \draw (1,2)--(4,2);
 \draw (1,3)--(4,3); 
 \draw (2,4)--(4,4); 

 \draw (1,1)--(1,3);
 \draw (2,1)--(2,4);
 \draw (3,1)--(3,4);
 \draw (4,2)--(4,4);

\node at (1.5,1.5){$C$};
\node at (2.5,1.5){$D$};
\node at (1.5,2.5){$E$};
\node at (2.5,2.5){$F$};
\node at (3.5,2.5){$G$};
\node at (2.5,3.5){$H$};
\node at (3.5,3.5){$I$};
   \end{tikzpicture}}
      \caption{$\PP_1=\PP\setminus \{A,B\}$}
   \end{subfigure}%
   \begin{subfigure}[t]{0.333 \textwidth}
   \centering
    \resizebox{0.4\textwidth}{!}{
  \begin{tikzpicture}
 \draw (3,2)--(4,2);
 \draw (2,3)--(4,3); 
 \draw (2,4)--(4,4);

 \draw (2,3)--(2,4);
 \draw (3,2)--(3,4);
 \draw (4,2)--(4,4);

\node at (3.5,2.5){$G$};
\node at (2.5,3.5){$H$};
\node at (3.5,3.5){$I$};
   \end{tikzpicture}}
    \caption{$\PP_1\setminus \{C,D,E,F\}$}
   \end{subfigure}
   \caption{A shortenable polyomino that is not well-shortenable }\label{fig:shornotwell}
\end{figure}

\end{Example}

In order to characterize the parallelogram polyominoes that are shortenable, we prove the following. 

\begin{Lemma}\label{lem:RRR}
Let $\PP$ be a parallelogram polyomino. Assume that $R_0$ has size $s\times t$ and its single rectangle $R$ has size $s'\times t'$ with $s'<s$ and $t' < t$. Then there exist $R',R'', \widetilde{R} \in \MM(\PP)$ as in Figure \ref{fig:rectproof}. 
\begin{figure}[H]
  \centering
    \resizebox{0.4\textwidth}{!}{
  \begin{tikzpicture}
 \draw (0,0)--(3,0);
 \draw (0,1)--(4,1);
 \draw (0,2)--(4,2);
\draw (2,3)--(4,3);

 \draw (0,0)--(0,2);
 \draw (2,0)--(2,3);
 \draw (3,0)--(3,3);
 \draw (4,1)--(4,3);

 \draw[pattern=north west lines] (0,1) rectangle (4,2);
 \draw[pattern=north east lines] (2,0) rectangle (3,3);
 
 \filldraw (0,0) circle (\rad) node [anchor=north]{$(0,0)$};
 
 \node at (1,0.5){$R$};
  \node at (-0.3,1){$R_0$};
  \node at (4.3,1.5){$R'$};
  \node at (2.5,-0.3){$R''$};
    \node at (3,3.3){$\widetilde{R}$};
   \end{tikzpicture}}
   \caption{}\label{fig:rectproof}
   
   \end{figure}%
\end{Lemma}
\begin{proof}
Let $\PP=(\SS_1,\SS_2)$. Since $s'<s$ and $t'<t$, then all cells of $R_0$ with lower left corner $(a,b)$ with either $s'\leq a$ or $t'\leq b$ belong to some other maximal rectangles of $\PP$ as well. Using the fact that $R_0 \in \MM(\PP)$, we observe that $\SS_2$ takes a north step at $(s,0)$. By using the assumptions $s'<s$ and $t'<t$ and $R$ is the single rectangle of $R_0$, we conclude that $\SS_2$ changes the direction from north to east at  $(s,t')$. Then the coordinates of $R'$ are determined by the next north turn of $\SS_2$. Similar argument on $\SS_1$ shows the existence of $R''$. The existence of $\tilde{R}$ is guaranteed by the fact that $\PP$ is a parallelogram polyomino and hence a distributive lattice, therefore the join of the diagonal corners of $R'$ and $R''$ must belong to $V(\PP)$. 

\end{proof}

In the following, we give a characterization of parallelogram polyominoes that are shortenable in terms of the size of the single rectangle of $R_0$.

\begin{Lemma}\label{lem:shorten}
Let $\PP$ be a parallelogram polyomino and assume $R_0$ has size  $s \times t$. Then $\PP$ is shortenable if and only if the single rectangle $R$ of $R_0$ has size $s'\times t'$ with either $s'=s$ or $t'=t$.
\end{Lemma}
\begin{proof}
By contraposition, assume that $R$ has size $s' \times t'$ with $s'<s$ and $t'<t$. From Lemma \ref{lem:RRR} there exist in $\PP$ the maximal rectangles in Figure \ref{fig:rectproof}. We consider the polyomino $\PP_1=\PP\setminus R_0$. We observe that $(s',t),(s,t')\in V(\PP_1)$ with $s'<s$ and $t'<t$, that is a contradiction to the fact that $\PP$ is parallelogram.\\
Conversely, assume that  $V(R)=[(0,0),(s',t)]$ with $s'<s$. Then using Proposition \ref{prop:parallel=simplelattice}, we obtain $\PP\setminus R_0$  is the parallelogram polyomino that corresponds to the sublattice $L\cap [(s',t),\max L]$.
\end{proof}

We now want to link the shortenability to the Gorensteinness. Hibi showed in \cite[page 105]{H} that given a distributive lattice $L$, the Hibi ring $K[L]$ is Gorenstein if and only if the poset $P$ of the join-irreducible elements of $L$ is pure, i.e. all of the maximal chains have the same length. Hence we look at the structure of the poset of the join-irreducible elements of parallelogram polyomino $\PP$ that we identify as a distributive lattice.

Let $H_0, H_1, \ldots, H_n$ be the maximal edge horizontal intervals of $\Pc$ and $V_0, V_1, \ldots, V_n$ be the maximal edge vertical intervals of $\Pc$. Note that $H_0\cap V_0=\{(0,0)\}=\min L$. Set $h_i=\min (H_i)$ for all $1\leq i \leq n$ and $v_j=\min (V_i)$ for all $1\leq j \leq m$ (see Figure \ref{fig:poset}).
Then $h_1\leq h_2\leq \ldots \leq h_n$ and $v_1 \leq v_2 \leq \ldots \leq v_m$ are two maximal chains of $P$.

\begin{figure}[H]
    \centering
    \begin{subfigure}[t]{0.45 \textwidth}
    \centering
    \resizebox{0.65\textwidth}{!}{
  \begin{tikzpicture}
 \draw (0,0)--(2,0);
 \draw (0,1)--(3,1);
 \draw (1,2)--(4,2);
 \draw (1,3)--(4,3); 
 \draw (2,4)--(4,4);

 \draw (0,0)--(0,1);
 \draw (1,0)--(1,3);
 \draw (2,0)--(2,4);
 \draw (3,1)--(3,4);
 \draw (4,2)--(4,4);

\filldraw (0,1) circle (\rad) node [anchor=east]{$h_1$};
\filldraw (1,2) circle (\rad) node [anchor=east]{$h_2$};
\filldraw (1,3) circle (\rad) node [anchor=east]{$h_3$};
\filldraw (2,4) circle (\rad) node [anchor=east]{$h_4$};

\node at (1,0){$\times$};
\node at (1,0)[anchor=north]{$v_1$};

\node at (2,0){$\times$};
\node at (2,0)[anchor=north]{$v_2$};

\node at (3,1){$\times$};
\node at (3,1)[anchor=north]{$v_3$};

\node at (4,2){$\times$};
\node at (4,2)[anchor=north]{$v_4$};
   \end{tikzpicture}}
   \caption{The join-irreducible elements are the minimum of the maximal horizontal and vertical edge intervals}
    \end{subfigure} \hfill%
    \begin{subfigure}[t]{0.45\textwidth}
    \centering 
    \begin{tikzpicture}
    \draw (0,0)--(0,3);
    \draw (1,0)--(1,3);
    
    \draw (0,0)--(1,2);
    \draw (0,1)--(1,3);
    
    \draw (1,0)--(0,1);
    \draw (1,1)--(0,3);

\filldraw (0,0) circle (\rad) node [anchor=east]{$h_1$};
\filldraw (0,1) circle (\rad) node [anchor=east]{$h_2$};
\filldraw (0,2) circle (\rad) node [anchor=east]{$h_3$};
\filldraw (0,3) circle (\rad) node [anchor=east]{$h_4$};

\filldraw (1,0) circle (\rad) node [anchor=west]{$v_1$};
\filldraw (1,1) circle (\rad) node [anchor=west]{$v_2$};
\filldraw (1,2) circle (\rad) node [anchor=west]{$v_3$};
\filldraw (1,3) circle (\rad) node [anchor=west]{$v_4$};
    \end{tikzpicture}
    \caption{The poset $P$ of join-irreducible elements}
    \end{subfigure}
     \caption{}\label{fig:poset}
\end{figure}

In \cite{EHQR}, the authors prove that an $L$-convex polyomino $\PP$ with derived sequence $(\PP_k)_{k=1, \ldots, t}$ for some $t$ is Gorenstein if and only if the bounding box of any $\PP_k$ is a square. For parallelogram polyominoes the latter condition is necessary but not sufficient, as shown in Figure \ref{fig:nongor}. The polyomino $\PP$ in Figure \ref{fig:nongor} is known to be non-Gorenstein from \cite[Theorem 4.2]{RR}, while $\PP$, $\PP_1$ and $\PP_2$ have square bounding boxes.

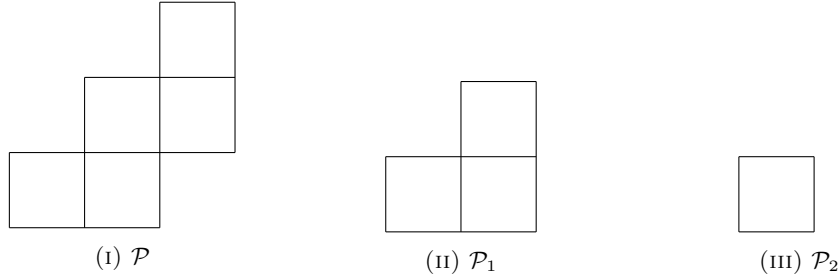
\begin{figure}[H]
\centering
\begin{subfigure}{0.3\textwidth}
\centering
\begin{tikzpicture}
\draw (0,0)--(2,0);
\draw (0,1)--(3,1);
\draw (1,2)--(3,2);
\draw (2,3)--(3,3);

\draw (0,0)--(0,1);
\draw (1,0)--(1,2);
\draw (2,0)--(2,3);
\draw (3,1)--(3,3);

\end{tikzpicture}
\caption{$\PP$}
\end{subfigure}%
\begin{subfigure}{0.3\textwidth}
\centering
\begin{tikzpicture}
\draw (0,0)--(2,0);
\draw (0,1)--(2,1);
\draw (1,2)--(2,2);

\draw (0,0)--(0,1);
\draw (1,0)--(1,2);
\draw (2,0)--(2,2);
\node at (1.5,3){};
\end{tikzpicture}
\caption{$\PP_1$}
\end{subfigure}%
\begin{subfigure}{0.3\textwidth}
\centering
\begin{tikzpicture}
\draw (0,0)--(1,0);
\draw (0,1)--(1,1);

\draw (0,0)--(0,1);
\draw (1,0)--(1,1);
\node at (1.5,3){};
\end{tikzpicture}
\caption{$\PP_2$}
\end{subfigure}%
\caption{An example of non-Gorenstein parallelogram polyomino with square bounding boxes}\label{fig:nongor}
\end{figure}

Next, we prove that a Gorenstein parallelogram polyomino is well-shortenable. 

\begin{Lemma}\label{GtoWS}
Let $\PP$ be a parallelogram polyomino. If $\PP$ is Gorenstein, then $\PP$ is well-shortenable.
\end{Lemma}
\begin{proof}
Let $\PP$ be Gorenstein. Then due to \cite[page 105]{H}, the poset $P$ of join-irreducible elements of $\PP$ is pure. Assume that $\PP$ is not shortenable. Then by using Lemma \ref{lem:shorten}, we obtain that if $R_0$ has size $s\times t$ with $t\leq s$ then the single rectangle of $R$ has size $r\times q$ with $V(R)=[(0,0),(r,q)]$  with $q\leq r$, $r<s$ and $q<t$. From Lemma \ref{lem:RRR}, we can find the maximal rectangle $R'$ with $V(R')=[(0,q),(u,t)]$ with $u > s$ as shown in Figure \ref{fig:vrht}. 
\begin{figure}[H]
  \centering
    \resizebox{0.4\textwidth}{!}{
  \begin{tikzpicture}
 \draw (0,0)--(3,0);
 \draw (0,1)--(4,1);
 \draw (0,2)--(4,2);

 \draw (0,0)--(0,2);
 \draw (2,0)--(2,1);
 \draw (3,0)--(3,2);
 \draw (4,1)--(4,2);
\filldraw (0,1) circle (\rad) node [anchor=east]{ {\scriptsize $h_q$}};
\filldraw (0,2) circle (\rad) node [anchor=east]{{\scriptsize   $h_t$}};

 \node at (2,0){$\times$};
  \node at (3,0){$\times$};
   \node at (3.3,1){$\times$};
 \node at (2,0) [anchor=north]{{\scriptsize $v_r$}};
  \node at (3,0) [anchor=north]{{\scriptsize $v_s$}};
    \node at (3.4,1) [anchor=north]{{\scriptsize $v_{s+1}$}};
    \node at (1,0.5){$R$};

    \node at (4.3,1.5){$R'$};
   \end{tikzpicture}}
   \caption{}\label{fig:vrht}
   
   \end{figure}%
We observe that $h_q$ and $v_{s+1}$ correspond to $(0,q)$ and $(s+1,q)$, respectively.   
The latter implies that $h_q \leq v_{s+1}$. We consider the following maximal chains of the poset $P$,
\[
 h_1 \leq h_2 \leq \ldots \leq h_n, \ \ \ h_1\leq h_2\leq \ldots \leq h_q \leq v_{s+1}\leq \ldots \leq v_n.
\]
 The first chain has length $n$ while the second one has length $n-s+q<n$ since $q<t\leq s$. This contradicts the Gorensteinness of $\PP$. Therefore, we conclude that $\PP$ is shortenable and hence, $ \PP_1=\PP\setminus R_0$ is a parallelogram polyomino. 

To show that $\PP$ is well-shortenable, it is enough to show that $\PP_1$ is Gorenstein. Indeed, if $\PP_1$ is Gorenstein then by following the previous argument, it is shortenable and the conclusion follows by applying the same argument. Let $P_1$ be the poset of the join-irreducible elements of $\PP_1$. Assume that the single rectangle $R$ of $R_0$ (in $\PP$) is such that $V(R)=[(0,0),(r,t)]\subset V(R_0)=[(0,0),(s,t)]$. Then $\min (\PP_1)=(r,t)$ and in $P$ we have $v_{r}\leq h_{t+1}$. If $\PP_1$ is not Gorenstein, we exhibit two chains in $\PP$ that have different lengths. Let  
\[
c_1\leq c_2 \leq \ldots \leq c_l, \ \ d_1\leq d_2 \leq \ldots \leq d_h
\]
be two chains with $c_i,d_j \in V(P_1)=\{v_{r+1},\ldots,v_n,h_{t+1},\ldots , h_n\}$ and $l\neq h$. If $c_1=d_1=h_{t+1}$, then $h_1\leq\ldots \leq h_t \leq c_1\leq c_2 \leq \ldots \leq c_l$ and $h_1\leq\ldots \leq h_t\leq d_1\leq d_2 \leq \ldots \leq d_h$ are two maximal chains of $P$ having different lengths, a contradiction to the Gorensteinness of $\PP$. Similar arguments hold for the case $c_1=d_1=v_{r+1}$. We are left with the case $c_1=h_{t+1}$ and $d_1=v_{r+1}$. Since $v_{r}\leq h_{t+1}$, then
\[
v_1\leq \ldots \leq v_{r} \leq c_1 \leq c_2 \leq \ldots \leq c_l, v_1\leq \ldots \leq v_{r} \leq d_1 \leq d_2 \leq \ldots \leq d_h
\]
are two chains of $P$ having lengths $r+l$ and $r+h$ and since $h\neq l$, then they have different lengths and $\PP$ is not Gorenstein, a contradiction. This shows that $\PP\setminus R_0$ is Gorenstein and hence shortenable. 
\end{proof}

In order to link the Gorensteinness with the $S$-property, we prove that a parallelogram polyomino with $S$-property is well shortenable. 
\begin{Lemma}\label{lem:singles}
Let $\PP$ be a parallelogram polyomino with S-property. Then $\PP$ is shortenable.
\end{Lemma}
\begin{proof}
 Let $S$ be the single square of $R_0$. Assume that $R_0$ has size $s \times t$ and $S$ has size $r \times r$ with $r<\min \{s,t\}$. From Lemma \ref{lem:RRR}, there exist some maximal rectangles $R',R''$ and $\widetilde{R}$ as in Figure \ref{fig:rectproof}. We observe that in this case $R',R''$ are contained in $R_0\cup \widetilde{R}$, that is they do not have single squares, and it is a contradiction to the fact that $\PP$ has the $S$-property. Therefore, either $r=s$ or $r=t$ and the conclusion follows from Lemma~\ref{lem:shorten}.
\end{proof}

\begin{Corollary}\label{cor:StoWS}
Let $\PP \subseteq [m,n]$ be a parallelogram polyomino with S-property, let $R_0,\ldots, R_l$  be the maximal rectangles of $\PP$ having single squares $S_0,S_1,\ldots S_l$ of sizes $t_1 \times t_1,\ldots, t_l \times t_l$, respectively. For any $i \in {1,\ldots, l}$ let $c_i=\sum_{j=1}^i t_j$. Then, we have  $V(S_{i})\cap V(S_{i+1})=(c_i,c_i)$ and $m=n=c_l$. Moreover $\PP$ is well-shortenable.
\end{Corollary}
\begin{proof}
From Lemma \ref{lem:singles} we have that $V(S_0)=[(0,0),(t_1,t_1)]$ and $\PP$ is shortenable. Let $\PP_1=\PP\setminus R_0$. From Lemma \ref{lem:singles} applied to $\PP_1$, we obtain that $S_1$ is such that $V(S_1)=[(t_1,t_1),(c_2,c_2)]$. We recursively consider the polyomino $\PP_i$ obtained from from $\PP_{i-1}$ by removing the rectangle $R_i$ and we obtain from Lemma \ref{lem:singles}.that $V(S_{i+1})=[(c_i,c_i),(c_{i+1},c_{i+1})]$. The polyomino $\PP_l$ is a square, that is $c_l=m=n$ and $\PP$ is well-shortenable.
\end{proof}

Now we prove the main theorem of this section.
\begin{Theorem}\label{thm:GorSprop}
Let $\PP$ be a parallelogram polyomino. The following are equivalent:
\begin{itemize}
\item[(i)] $\PP$ is Gorenstein;
\item[(ii)] $\PP$ has the $S$-property.
\end{itemize}
\end{Theorem}
\begin{proof}
(i)$\Rightarrow$(ii). From Lemma \ref{GtoWS}, we have that $\PP$ is well shortenable. Moreover, from the proof of Lemma \ref{GtoWS} it arises that all of the polyominoes in the derived sequence are Gorenstein, that is they have square bounding boxes due to the pureness of the poset. In particular, $\PP \subset  [(0,0),(n,n)]$.

First, we show that the single rectangle $R$ of $R_0$ is a square. Let $P$ be the poset of the join-irreducible elements of $\PP$. Assume that
\[
V(R)=[(0,0),(s,t)]\subset V(R_0)=[(0,0),(q,t)]
\]
with $s \neq t$. Hence $\min V(\PP_1)=(s,t)$ and in $P$ we have $v_{s}\leq h_{t+1}$. This gives that the two chains
\[
v_1 \leq \ldots \leq v_s\leq v_{s+1} \leq \ldots \leq v_n, v_1 \leq \ldots \leq v_s\leq h_{t+1} \leq \ldots \leq h_n
\]
have different lengths and this is a contradiction to the assumption that $P$ is Gorenstein. That is $s=t$. 
Furthermore, we claim that there exists a unique maximal rectangle $R_1$ containing $\widetilde{R}= R_0\setminus R$, namely $[(s,0),(q,s)]$. Let $R_1'=[(a,b),(c,d)]$ be a maximal rectangle such that $R_1' \cap \widetilde{R}\neq \varnothing$, that is $s \leq a \leq q$ and $b<s$. From the property of parallelogram polyominoes, we also obtain that $b\geq 0$. If $b>0$, then the rectangle $[(a,0),(c,d)]$ is a rectangle containing $R_1'$, contradicting its maximality. That is, we have $b=0$. We observe that $c\leq q$, otherwise the rectangle $[(0,b),(c,s)]$ is a maximal rectangle having non-empty intersection with $R$, contradiction. Moreover $d>s$, otherwise $R_1'\subseteq \widetilde{R}$.
The latter implies that all of the maximal rectangle having non-empty intersection with $\widetilde{R}$ have lower left corner on the edge interval $[(s,0),(q,0)]$. Then, there exists a unique maximal rectangle $R_1$ with vertices $[(s,0),(q,u)]$ where $u$ is the minimum of the heights of such rectangles.
 We now show that $R_1$ has a single rectangle. If this is not the case, then there exists a maximal rectangle $R_2$ such that $R_1 \subseteq R_0 \cup R_2$ and $V(R_2)=[(s,s),(a,b)]$ with $a > q$, hence $h_s \leq v_{q+1}$ in $P$. This implies that 
\[
h_1 \leq \ldots \leq h_s\leq v_{q+1} \leq \ldots \leq v_n, h_1 \leq \ldots \leq h_n
\]
are two chains having lengths $n-q+s$ and $n$, respectively. Since $s<q$, then $n-q+s<n$, contradicting the Gorensteinnes of $\PP$. In particular this implies that any maximal rectangle has a single rectangle. By using a similar technique on any polyomino of the derived sequence we obtain that all of the rectangles of $\PP$ have a single square.

(ii)$\Rightarrow$(i). We assume that $\PP$ has the $S$-property. To have the Gorensteinness, we have to prove that for any edge in the Hasse diagram of the poset $\PP$ of the form $v_s \to h_{t+1}$ (or $h_c \to v_{d+1}$), we have $s=t$ (or $c=d$). We follow the notation of Corollary \ref{cor:StoWS}. From the latter result we obtain that if $S_0$ has size $t_0 \times t_0$ and $S_1$ has size $t_1 \times t_1$. Then $R_0$ has either size $t_0 \times (t_0 +t_1)$ or $(t_0 +t_1) \times t_0$, that is either $h_{t_0}\to v_{t_0 +1}$ or $v_{t_0}\to h_{t_0 +1}$.  
Since $\PP$ is well-shortenable, we inductively apply the same argument to find that for any $k \in \{1,\ldots l\}$ either $h_{c_k}\to v_{c_k +1}$ or $v_{c_k}\to h_{c_k +1}$. \\ Moreover, assume that  $h_r \to v_s$ is an edge of the poset $P$ such that $c_{k-1}+1 \leq r < c_{k}$ for some $k$. It follows that $s>c_{k}+1$ and there exists a maximal rectangle in $\MM(\PP)$ of size $a \times b $ with $b = c_{k}-r$  that has non-empty intersection with $S_k$. This leads to a contradiction to the fact that $S_k$ is single.
\end{proof}

 Now, we give a description of Gorenstein parallelogram polyominoes in terms of the 2-colored Motzkin paths. To do this, we first recall the well-known bijection between the parallelogram polyominoes and 2-colored Motzkin paths, see \cite{DV}.  Let $(a,b) \in \NN$. Then 

\begin{enumerate}
    \item the edge $\{(a,b), (a+1,b+1)\}$ is called a  {\em rise} step, 
    \item  the edge $\{(a,b), (a+1,b-1)\}$ is called a {\em fall} step,
     \item  the edge $\{(a,b), (a+1,b)\}$ is called a {\em east} step or a {\em horizontal} step.
\end{enumerate}
A 2-colored Motzkin path 
\[
\mathcal{M}: (0,0)=(a_0,b_0), (a_1, b_1), \ldots, (a_n,b_n)=(n,0)
\] 
in $\NN \times \NN$ is a path that never passes below the $x$-axis and consists of rise steps, fall steps and two types of horizontal steps that are called $\alpha$-colored horizontal steps and $\beta$-colored horizontal steps.
Let $\Pc$ be a parallelogram polyomino determined by $(\mathcal{S}_1,\mathcal{S}_2)$ such that $\mathcal{S}_1$ and $\mathcal{S}_2$ intersect at $(0,0)$ and $(m,n)$. Then $\Pc$ can be encoded in a unique 2-colored Motzkin path $\MM_\PP$ as described in the following algorithm given in \cite{DNPR}.

Each north-east path in $\NN \times \NN$ of length $n$ can be identified as a binary sequence with 0 representing an east step and 1 representing a north step.   Let $\Pc=(\SS_1, \SS_2)$ be a parallelogram polyomino and let $u(\Pc)$ be the binary tuple representing $\SS_1$ and $\ell(\Pc)$ be the binary tuple representing $\SS_2$. Create a matrix $M$ with $u(\Pc)$ as its first row and $\ell(\Pc)$ as its second row.  Then $M$ can be encoded as a Motzkin path by the coding:

\begin{equation}
        \begin{aligned}[a]
        &\binom{1}{0} \mapsto  \text{rise step}  &  &\binom{0}{1} \mapsto  \text{fall step}\\
        &\binom{1}{1} \mapsto  \text{$\alpha$-colored horizontal step}  &  &\binom{0}{0} \mapsto  \text{$\beta$-colored horizontal  step} 
        \end{aligned}
\label{eq:coding}
\end{equation}

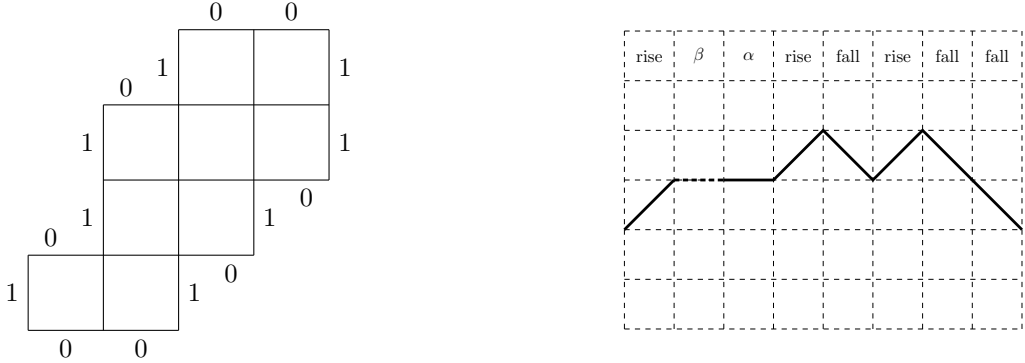
\begin{figure}[H]
    \centering
    \begin{subfigure}{0.45\textwidth}
      \begin{tikzpicture}
 \draw (0,0)--(2,0);
 \draw (0,1)--(3,1);
 \draw (1,2)--(4,2);
 \draw (1,3)--(4,3); 
 \draw (2,4)--(4,4);

 \draw (0,0)--(0,1);
 \draw (1,0)--(1,3);
 \draw (2,0)--(2,4);
 \draw (3,1)--(3,4);
 \draw (4,2)--(4,4);
 
 \node at (0.5,0) [anchor=north]{$0$};
 \node at (1.5,0) [anchor=north]{$0$};
 \node at (2.7,1) [anchor=north]{$0$};
 \node at (3.7,2) [anchor=north]{$0$};
 
 \node at (2,0.5) [anchor=west]{$1$};
 \node at (3,1.5) [anchor=west]{$1$};
 \node at (4,2.5) [anchor=west]{$1$};
 \node at (4,3.5) [anchor=west]{$1$};
 
 \node at (0,0.5) [anchor=east]{$1$};
 \node at (1,1.5) [anchor=east]{$1$};
 \node at (1,2.5) [anchor=east]{$1$};
 \node at (2,3.5) [anchor=east]{$1$};
 
 \node at (0.3,1) [anchor=south]{$0$};
 \node at (1.3,3) [anchor=south]{$0$};
 \node at (2.5,4) [anchor=south]{$0$};
 \node at (3.5,4) [anchor=south]{$0$};
   \end{tikzpicture}
    \end{subfigure}\hfill %
    \begin{subfigure}{0.45\textwidth}
     \resizebox{0.8\textwidth}{!}{
    \begin{tikzpicture}
    \draw[dashed] (0,0)--(8,0);
    \draw[dashed] (0,1)--(8,1);
    \draw[dashed] (0,2)--(8,2);
    \draw[dashed] (0,3)--(8,3);
    \draw[dashed] (0,4)--(8,4);
    \draw[dashed] (0,5)--(8,5);
    \draw[dashed] (0,6)--(8,6);
    
    \draw[dashed] (0,0)--(0,6);
    \draw[dashed] (1,0)--(1,6);
    \draw[dashed] (2,0)--(2,6);
    \draw[dashed] (3,0)--(3,6);
    \draw[dashed] (4,0)--(4,6);
    \draw[dashed] (5,0)--(5,6);
    \draw[dashed] (6,0)--(6,6);
    \draw[dashed] (7,0)--(7,6);
    \draw[dashed] (8,0)--(8,6);
    
    \draw[ultra thick] (0,2)--(1,3);
    \draw[ultra thick, densely dashed] (1,3)--(2,3);
    \draw[ultra thick] (2,3)--(3,3)--(4,4)--(5,3)--(6,4)--(8,2);
    
    \node at (0.5,5.5){rise};
    \node at (1.5,5.5){$\beta$};
    \node at (2.5,5.5){$\alpha$};
    \node at (3.5,5.5){rise};
    \node at (4.5,5.5){fall};
    \node at (5.5,5.5){rise};
    \node at (6.5,5.5){fall};
    \node at (7.5,5.5){fall};
    \end{tikzpicture}}
    \end{subfigure}
    \caption{A parallelogram polyomino with its Motzkin path}
    \label{fig:motzexa}
\end{figure}

For example, Figure~\ref{fig:motzexa} shows a  parallelogram polyomino and the associated 2-colored Motzkin path. The $\beta$-colored horizontal steps are shown as dashed lines and the $\alpha$-colored steps are shown as normal lines.  
We observe that
\[
u(\PP)= 1 0 1 1 0 1 0 0  \ \ \ \
\ell(\PP)= 0 0 1 0 1 0 1 1 
\]
The associated matrix $M$ of $\Pc$ described above is :
\[
\begin{pmatrix}
1 &0 &1 &1 &0 &1 &0 &0\\
0 &0 &1 &0 &1 &0 &1 &1 \\
\end{pmatrix}
\]

We need the following terminologies to describe Gorenstein parallelogram polyominoes in terms of their associated 2-colored Motzkin paths. 

\begin{Definition}
\begin{enumerate}
\item Let $\SS:s_1,\ldots, s_l$ be a north-east path. A sequence of consecutive north steps (resp. east steps) $s_i,\ldots, s_{i+k}$ makes a \emph{maximal block} of length $k$ in $\SS$ if either $i=1$ or $s_{i-1}$ is an east step (resp. north step), and either $i+k=l$ or $s_{i+k+1}$ is a north step (resp. east step). Note that in $\SS$, a maximal block of length $k$ of consecutive north steps (resp. east steps) corresponds to a maximal block of $k$ 1s (resp. 0s) in its binary representation.

\item Let $\PP=(\SS_1,\SS_2)$ be a parallelogram polyomino. A sequence of consecutive elements $s_1, s_2,\ldots, s_{l}$ of $\SS_1$ (resp. $\SS_2$) is called a \emph{maximal NE-block} if there exists $i \in \{1,\ldots l\}$ such that $s_1 \ldots s_i$ is a maximal block of north steps (resp. east steps) and $s_{i+1} \ldots s_l$ is a maximal block of east steps (resp. north steps).
\end{enumerate}
\end{Definition}

For example, for the parallelogram polyomino given in Figure \ref{fig:motzexa}, the binary representation of $\SS_1$ is $u(\PP): 10110100$.  The maximal NE-block in $\SS_1$ are $10$, $110$ and $100$. The NE-blocks in $\SS_1$ determine the corners in $\SS_1$. Similarly, the maximal NE-block in $\SS_2$ are $001$, $01$, and  $011$ and the NE-blocks in $\SS_2$ determine the corners in $\SS_2$. We emphasize that in $\SS_1$ each NE-block starts with a north step, while an NE-block in $\SS_2$ starts with an east step.


\begin{Theorem}\label{thm:motpath}
Let $\PP=(\SS_1,\SS_2)$ be a parallelogram polyomino. $\PP$ has the $S$-property if and only if the following conditions hold:
\begin{enumerate}
    \item in $\SS_1$, each maximal block of length $k$ of consecutive north steps is followed by a maximal block of length $k$ of consecutive east steps. 
    \item in $\SS_2$, each maximal block of length $k$ of consecutive east steps is followed by a maximal block of length $k$ of consecutive north steps.
\end{enumerate}
\end{Theorem}
\begin{proof}
Assume that $\PP$ has the $S$-property. We need to show that $\PP$ satisfies the conditions (1) and (2). We proceed by induction on the total number $l$ of maximal rectangles of $\PP$.\\
If $l=1$, then $\PP$ itself is a rectangle. Using the assumption that $\PP$ has the $S$-property, we see that $\PP$ is in fact a square of size $t\times t$. This shows that binary representations of $\SS_1$ and $\SS_2$ are given by

\[
u(\PP): \ \underbrace{11\ldots1}_{t \mbox{ \scriptsize times}} \underbrace{00\ldots0}_{t \mbox{ \scriptsize times}}, \ \ \ \ell (\PP) : \  \underbrace{00\ldots0}_{t \mbox{ \scriptsize times}} \underbrace{11\ldots1}_{t \mbox{ \scriptsize times}}
\]
as claimed.

Now assume that $l\geq 2$ and the assertion is true for any parallelogram polyomino with $l-1$ maximal rectangles. Let the size of $R_0$ be $s \times t$. Assume that $t<s$, and the case when $t>s$ can be discussed in a similar way. The assumption that $\Pc$ has the $S$-property together with Lemma~\ref{lem:shorten} and \ref{lem:singles} shows that the single square $S$ of $R_0$ has size $t\times t$. Consider the parallelogram polyomino $\PP'= \PP\setminus R_0$ given by some paths $(\SS_1', \SS_2')$. Then $\PP'$ has the $S$-property, too. We observe that since $S$ is single square, the path $\SS_1$ is of the form
\[
\underbrace{11\ldots1}_{t \mbox{ \scriptsize times}}\underbrace{00\ldots0}_{t \mbox{ \scriptsize times}}u(\PP').
\]
 where $u(\PP')$ is the binary representation of $\SS'_1$. By using the inductive hypothesis on $\PP'$ we conclude that $\PP$ satisfies the condition (1). Moreover, again by using the inductive hypothesis on $\PP'$, we see that $\SS_2'$ satisfies condition (2). That is, the binary representation $\ell(\PP')$ of $\SS_2'$ starts with a block of $s-t$ $0$s followed by a block of $s-t$ $1$s, in particular
\[
\underbrace{00\ldots0}_{s-t \mbox{ \scriptsize times}} \beta.
\]
Hence $\ell(\PP)$ is given by
\[
\underbrace{00\ldots0}_{s \mbox{ \scriptsize times}} \underbrace{11\ldots1}_{t \mbox{ \scriptsize times}}\beta
\]
This shows that $\SS_2$ satisfies the condition (2).

To prove the converse, assume that $\PP$ satisfies the conditions (1) and (2). We need to show that $\PP$ has the S-property. We proceed by induction on the total number $e\geq 2$ of maximal NE-blocks in $\SS_1$ and $\SS_2$. In other words, we apply the induction on the total number of corners in $\SS_1$ and $\SS_2$. 

For $e=2$, from the conditions (1) and (2) we get that in $\SS_1$ (resp. $\SS_2$) the maximal NE-block has size $2t$ for some $t \in \NN$. More precisely, $\SS_1$ (resp. $\SS_2$) has a maximal block of $t$ north-steps (resp. east-steps) followed by a block of $t$ east-steps (resp. north-steps). Then the binary representations of $\SS_1$ and $\SS_2$ are
\[
u(\PP): \underbrace{11\ldots1}_{t \mbox{ \scriptsize times}} \underbrace{00\ldots0}_{t \mbox{ \scriptsize times}} \mbox{ and } \ell(\PP): \underbrace{00\ldots0}_{t \mbox{ \scriptsize times}} \underbrace{11\ldots1}_{t \mbox{ \scriptsize times}}
\]
and the polyomino is a square. 

Now, let $e\geq 3$ and assume that any parallelogram polyomino having a total number of maximal NE-blocks equal to $e-1$ has the $S$-property. Let $\PP'=\PP \setminus R_0$.  It follows from Lemma~\ref{lem:singles} that $\PP'$ is a parallelogram polyomino. Set $\PP'=(\SS_1', \SS_2')$. To prove that $\PP$ has the $S$-property, it is enough to show that $R_0$ has a single square and that $\PP'$ has the $S$-property. In particular, we prove that $\SS_1'$ and $\SS_2'$ satisfy conditions (1) and (2), respectively. Then the conclusion will follow by using inductive hypothesis on $\PP'$ and the existence of single square in $R_0$. 

If $R_0$ has size $s\times t$ with $t <s$, then $u(\PP)$ begins with a maximal block of $t$ $1$s and by using condition (1), there is a maximal block of $t$ $0s$ following it. Therefore, $u(\PP)$ is of the following form
\[
\underbrace{11\ldots1}_{t \mbox{ \scriptsize times}}\underbrace{00\ldots0}_{t \mbox{ \scriptsize times}}u(\PP').
\]
This shows that $R_0$ has a single square of size $t \times t$ and $\SS'_1$ satisfies the condition (1). Moreover, by using the assumption that $\PP$ satisfies condition (2), we obtain that $\ell(\PP)$ starts with a maximal block of $s$ 0s followed by a maximal block of $s$ 1s. We write 
\[
\underbrace{00\ldots0}_{s \mbox{ \scriptsize times}} \underbrace{11\ldots1}_{s \mbox{ \scriptsize times}}\beta.
\]
 where $\beta$ is binary sequence consistent with condition (2). Then $\ell(\PP')$ takes the following form
\[
\underbrace{00\ldots0}_{s-t \mbox{ \scriptsize times}}\ \ \underbrace{11\ldots1}_{s-t \mbox{ \scriptsize times}}\beta.
\]
which shows that $\PP'$ satisfies the condition (2).  Moreover, the total number of maximal NE-blocks in $\PP'$ is $e-1$ (because one maximal NE-block is at the beginning of $\SS_1$). By using the inductive hypothesis, we conclude that $\PP'$ has the $S$-property. Then, it follows that $\PP$ has the $S$-property as well.
\end{proof}

With the help of Theorem~\ref{thm:GorSprop}, to be able to describe Motzkin paths associated with Gorenstein parallelogram polyominoes, it is enough to see the impact of conditions (1) and (2) of Theorem~\ref{thm:motpath} on the associated Motzkin paths. Let $\PP=(\SS_1, \SS_2)$ be a Gorenstein parallelogram polyomino with associated Motzkin path $\MM_\PP$. Note that in $\SS_1$, a maximal block of length $k$ of consecutive north steps corresponds to a combination of $k$ rise and $\alpha$-colored horizontal steps in $\MM_\PP$. Indeed, this combination of rise and $\alpha$-colored horizontal steps in $\MM_\PP$ is maximal in a sense that it is followed by either a fall or a $\beta$-colored horizontal step. Similarly, a maximal block of length $k$ of consecutive east steps corresponds to a maximal block of a combination of $k$ fall and $\beta$-colored horizontal steps in $\MM_\PP$. Hence, the condition (1) of Theorem~\ref{thm:motpath} translates as: in $\MM_\PP$ each maximal block of a combination of $k$ rise and $\alpha$-colored horizontal steps must be followed by a maximal block of a combination of $k$ fall and $\beta$-colored horizontal steps. 

To translate condition (2) for $\MM_{\PP}$, we consider the reflection of $\MM_{\PP}$ through the $x$-axis. We denote this reflection by $\overline{\MM_\PP}$. The reflection $\overline{\MM_\PP}$ corresponds to the coding given in (\ref{eq:coding}) applied to the matrix that contains $\ell(\PP)$ as first row and $u(\PP)$ as the second row. Then the condition (2) of Theorem~\ref{thm:motpath} translates as: in $\overline{\MM_\PP}$ each maximal block of a combination of $k$ fall and $\beta$-colored horizontal steps must be followed by a maximal block of a combination of $k$ rise and $\alpha$-colored horizontal steps. We formulate this discussion in the following corollary.

\begin{Corollary}\label{cor:Motz}
Let $\PP=(\SS_1,\SS_2)$ be a parallelogram polyomino with associated Motzkin path $\MM_\PP$. Let $\overline{\MM_\PP}$ be the reflection of $\MM_\PP$ through $x$-axis.  Then $\PP$ is Gorenstein if and only if  the following conditions hold: 
\begin{enumerate}
    \item in $\MM_\PP$ each maximal block of a combination of $k$ rise and $\alpha$-colored horizontal steps must be followed by a maximal block of a combination of $k$ fall and $\beta$-colored horizontal steps;
    \item in $\overline{\MM_\PP}$ each maximal block of a combination of $k$ fall and $\beta$-colored horizontal steps must be followed by a maximal block of a combination of $k$ rise and $\alpha$-colored horizontal steps.
\end{enumerate}
\end{Corollary}

We give an illustration of Corollary~\ref{cor:Motz} in the following example. 
\begin{Example}
The Figure~\ref{fig:GorMotz} shows a Gorenstein parallelogram polyomino. The associated Motzkin path $\MM_\PP$ is shown on the left side and its reflection through $x$-axis is shown on the right side. The Motzkin path $\MM_\PP$ and its reflection satisfy the conditions (1) and (2) of Corollary~\ref{cor:Motz}.
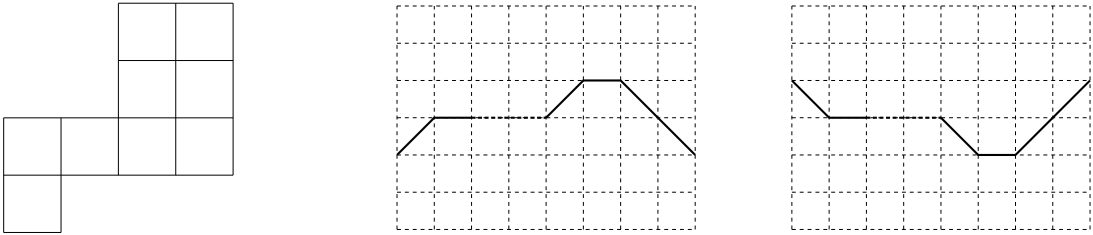
\begin{figure}[H]
    \centering
    \begin{subfigure}{0.3 \textwidth}
    \resizebox{0.7\textwidth}{!}{
    \begin{tikzpicture}
  
   \draw (0,0)--(1,0);
    \draw (0,1)--(4,1);
    \draw (0,2)--(4,2);
    \draw (2,3)--(4,3);
    \draw (2,4)--(4,4);

    \draw (0,0)--(0,2);
    \draw (1,0)--(1,2);
    \draw (2,1)--(2,4);
    \draw (3,1)--(3,4);
    \draw (4,1)--(4,4);
    \end{tikzpicture}}

    \end{subfigure}\hfill%
    \begin{subfigure}{0.3 \textwidth}
     \resizebox{0.9\textwidth}{!}{
    \begin{tikzpicture}
     \draw[dashed] (0,0)--(8,0);
    \draw[dashed] (0,1)--(8,1);
    \draw[dashed] (0,2)--(8,2);
    \draw[dashed] (0,3)--(8,3);
    \draw[dashed] (0,4)--(8,4);
    \draw[dashed] (0,5)--(8,5);
    \draw[dashed] (0,6)--(8,6);
    
    \draw[dashed] (0,0)--(0,6);
    \draw[dashed] (1,0)--(1,6);
    \draw[dashed] (2,0)--(2,6);
    \draw[dashed] (3,0)--(3,6);
    \draw[dashed] (4,0)--(4,6);
    \draw[dashed] (5,0)--(5,6);
    \draw[dashed] (6,0)--(6,6);
    \draw[dashed] (7,0)--(7,6);
    \draw[dashed] (8,0)--(8,6);
    
    \draw[ultra thick] (0,2)--(1,3)--(2,3);
    \draw[ultra thick, densely dashed] (2,3)--(4,3);
    \draw[ultra thick] (4,3)--(5,4)--(6,4)--(8,2);
    
    \end{tikzpicture}}
    \end{subfigure}\hfill%
    \begin{subfigure}{0.3 \textwidth}
     \resizebox{0.9\textwidth}{!}{
    \begin{tikzpicture}
     \draw[dashed] (0,0)--(8,0);
    \draw[dashed] (0,1)--(8,1);
    \draw[dashed] (0,2)--(8,2);
    \draw[dashed] (0,3)--(8,3);
    \draw[dashed] (0,4)--(8,4);
    \draw[dashed] (0,5)--(8,5);
    \draw[dashed] (0,6)--(8,6);
    
    \draw[dashed] (0,0)--(0,6);
    \draw[dashed] (1,0)--(1,6);
    \draw[dashed] (2,0)--(2,6);
    \draw[dashed] (3,0)--(3,6);
    \draw[dashed] (4,0)--(4,6);
    \draw[dashed] (5,0)--(5,6);
    \draw[dashed] (6,0)--(6,6);
    \draw[dashed] (7,0)--(7,6);
    \draw[dashed] (8,0)--(8,6);
    
    \draw[ultra thick] (0,3+1)--(1,2+1)--(2,2+1);
    \draw[ultra thick, densely dashed] (2,2+1)--(4,2+1);
    \draw[ultra thick] (4,2+1)--(5,1+1)--(6,1+1)--(8,3+1);
    
    \end{tikzpicture}}
    \end{subfigure}\hfill%
    \caption{A Gorenstein parallelogram polyomino satisfying conditions (1) and (2) of Corollary \ref{cor:Motz}  }
    \label{fig:GorMotz}
\end{figure}

The Figure~\ref{fig:nonGorMotz} shows a non-Gorenstein parallelogram polyomino. The associated Motzkin path $\MM_\PP$ is shown on the left side and its reflection through $x$-axis is shown on the right side. The Motzkin path $\MM_\PP$ fails the condition (1) of Corollary~\ref{cor:Motz}. However, its reflection satisfies the condition (2) of Corollary~\ref{cor:Motz}.
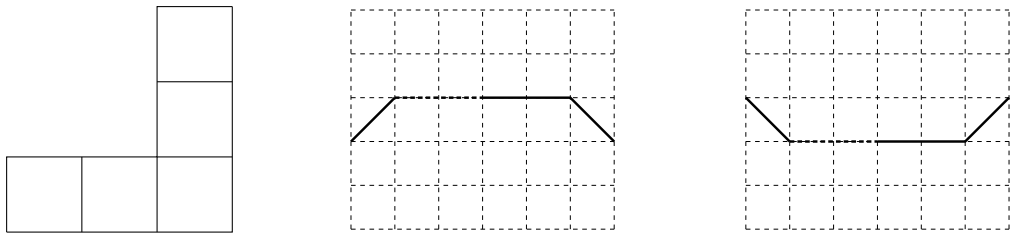
\begin{figure}[H] 
    \centering
    \hfill
    \begin{subfigure}{0.3 \textwidth}
    \begin{tikzpicture}
    \draw (0,0)--(3,0)--(3,3);
    \draw (0,0)--(0,1)--(2,1)--(2,3)--(3,3);
    \draw (1,0)--(1,1);\draw (2,0)--(2,1);\draw (2,1)--(3,1);\draw (2,2)--(3,2);
    
    
    \end{tikzpicture}

    \end{subfigure}%
    \begin{subfigure}{0.3 \textwidth}
     \resizebox{0.8\textwidth}{!}{
    \begin{tikzpicture}
     \draw[dashed] (0,0)--(6,0);
    \draw[dashed] (0,1)--(6,1);
    \draw[dashed] (0,2)--(6,2);
    \draw[dashed] (0,3)--(6,3);
    \draw[dashed] (0,4)--(6,4);
    \draw[dashed] (0,5)--(6,5);

    \draw[dashed] (0,0)--(0,5);
    \draw[dashed] (1,0)--(1,5);
    \draw[dashed] (2,0)--(2,5);
    \draw[dashed] (3,0)--(3,5);
    \draw[dashed] (4,0)--(4,5);
    \draw[dashed] (5,0)--(5,5);
    \draw[dashed] (6,0)--(6,5);
    
    \draw[ultra thick] (0,2)--(1,3);
    \draw[ultra thick, densely dashed] (1,3)--(3,3);
    \draw[ultra thick] (3,3)--(5,3)--(6,2);
    
    \end{tikzpicture}}
    \end{subfigure}\hfill%
    \begin{subfigure}{0.3 \textwidth}
     \resizebox{0.8\textwidth}{!}{
    \begin{tikzpicture}
     \draw[dashed] (0,0)--(6,0);
    \draw[dashed] (0,1)--(6,1);
    \draw[dashed] (0,2)--(6,2);
    \draw[dashed] (0,3)--(6,3);
    \draw[dashed] (0,4)--(6,4);
    \draw[dashed] (0,5)--(6,5);

    \draw[dashed] (0,0)--(0,5);
    \draw[dashed] (1,0)--(1,5);
    \draw[dashed] (2,0)--(2,5);
    \draw[dashed] (3,0)--(3,5);
    \draw[dashed] (4,0)--(4,5);
    \draw[dashed] (5,0)--(5,5);
    \draw[dashed] (6,0)--(6,5);
    
    \draw[ultra thick] (0,3)--(1,2);
    \draw[ultra thick, densely dashed] (1,2)--(3,2);
    \draw[ultra thick] (3,2)--(5,2)--(6,3);
    \end{tikzpicture}}
    \end{subfigure}\hfill%
    \caption{A non-Gorenstein parallelogram polyomino satisfying condition (2) of Corollary \ref{cor:Motz}  }
    \label{fig:nonGorMotz}
\end{figure}

\end{Example}

\noindent \textbf{Data availability statement:} The source code to produce the datasets analysed during the current study are available in the second author's repository, \cite{QRR}.


\begin{thebibliography}{}
\bibitem{ABG}{J.C. Aval, F. Bergeron, A. Garsia},
\textit{Combinatorics of labelled parallelogram polyominoes},  J. Combin. Theory Ser. A, 132 (2015), 32--57.


\bibitem{BGS} { A. Bj\"orner, A.M. Garsia, R. Stanley} \textit{An Introduction to Cohen-Macaulay Partially Ordered Sets},
In: Rival I. (eds) Ordered Sets. NATO Advanced Study Institutes Series (Series C Ñ Mathematical and Physical Sciences), vol 83. (1982), Springer, Dordrecht. 


\bibitem{B} {G. Birkhoff}, 
\textit{Lattice Theory}, 3rd. ed., Amer. Math. Soc. Colloq. Publ. No. 25, Amer. Math. Soc. Providence, R. L, 1967.

\bibitem{CN}{A. Corso, U. Nagel}, 
\textit{Monomial and toric ideals associated to Ferrers graphs}, 
\newblock Trans. Amer. Math. Soc. 361, 1371--1395, (2009).

\bibitem{DV}{M. P. Delest, G. Viennot},  \textit{Algebraic Languages and Polyominoes Enumeration}, Theoret. Comput. Sci. 34, (1984), 169--206.

\bibitem{DNPR}{A. Del Lungo, M. Nivat, R. Pinzani,  S. Rinaldi},
\textit{A bijection for the total area of parallelogram polyominoes},
Discrete Appl. Math. 144 (3), (2004), 291--302.

\bibitem{EHQR}{V. Ene, J. Herzog, A. A. Qureshi, F. Romeo},
\textit{Regularity and the Gorenstein property of $L$-convex polyominoes},
\newblock Electron. J. Combin. 28 (1) (2021), 1--23.

\bibitem{EQR} V. Ene, A. A. Qureshi, A. Rauf, \textit {Regularity of join-meet ideals of distributive lattices}, Electron. J. Combin. 20 (3) (2013), P-20.

\bibitem{GG}{C. D. Godsil, I. Gutman},
\textit{Some remarks on matching polynomials and its zeros}, 
\newblock Croatica Chemica Acta, 54, 53--59, (1981).

\bibitem{Go}{S. W. Golomb},
\newblock \textit{Polyominoes, puzzles, patterns, problems, and packagings}, Second edition,
\newblock Princeton University press, 1994.

\bibitem{M2}{D. Grayson, M. Stillman},
\newblock\textit{Macaulay2, a software system for research in algebraic geometry}, \url{http://www.math.uiuc.edu/Macaulay2/}.

\bibitem{HHO}{J. Herzog, T. Hibi, H. Ohsugi}, 
\newblock \textit{Binomial ideals}, Graduate Texts in Math. 279, Springer, Cham,  (2018). 

\bibitem{HQS}{J. Herzog, A.A. Qureshi, A. Shikama},
\textit{Gr\"obner bases of balanced polyominoes},
Math. Nachr., 288, 775--783, (2015).

\bibitem{HM}{J. Herzog,  S. Saeedi  Madani},
\textit{The coordinate ring of a simple polyomino},
\newblock  Illinois J. Math.,   58,  981--995, (2014).

\bibitem{H}{T. Hibi}, \textit{Distributive lattices, affine semigroup rings and algebras with straightening laws}, 
\newblock Commutative Algebra and Combinatorics (M. Nagata and H. Matsumura, Eds.), Adv. Stud.
Pure Math. 11, North Holland, Amsterdam, 1987, pp. 93--109.

\bibitem{KV}{M. Kummini, D. Veer},
\newblock \textit{The $h$-polynomial and the rook polynomial of some polyominoes},
\newblock preprint arXiv:2110.14905.


 \bibitem{MRR}{C. Mascia, G. Rinaldo, F. Romeo},
\newblock \textit{Primality of multiply connected polyominoes}, accepted in Illinois J. Math, (2020).

\bibitem{MRRsite}{C. Mascia, G. Rinaldo, F. Romeo},
\newblock \textit{Primality of polyominoes},
 \url{http://www.giancarlorinaldo.it/polyominoes-primality.html}
 
\bibitem{Qu}{A. A. Qureshi},
\textit{Ideals generated by 2-minors, collections of cells and stack polyominoes},
\newblock J. Algebra, 357, 279--303, (2012).

\bibitem{QRR}{A. A. Qureshi, G. Rinaldo, F. Romeo},
\newblock \textit{Hilbert series of simple polyominoes}, \\
\url{http://www.giancarlorinaldo.it/hilbert-series-of-simple-polyominoes.html}
 
\bibitem{QSS}{A. A. Qureshi, T.  Shibuta, A.  Shikama} 
\textit{Simple polyominoes are prime},
\newblock J.  Commut.  Alg.,  9,  413--422, (2017). 

\bibitem{RR}{G. Rinaldo, and F. Romeo}, 
\textit{Hilbert Series of simple thin polyominoes},
\newblock J. Algebr. Comb., 54, 607--624 (2021).

\bibitem{Ri}{J. Riordan},
\textit{An introduction to combinatorial analysis},
\newblock Wiley Publications in Mathematical Statistics. John Wiley \& Sons, Inc., New York; Chapman \& Hall, Ltd., London, (1958).

\bibitem{Vi}{R. Villarreal}, 
\newblock \textit{Monomial algebras}, Second edition,
\newblock Taylor and Francis, CRC Press, (2015).
\end{thebibliography}
\end{document}